\documentclass[oneside,12pt]{article}

\usepackage{amsmath}
\usepackage{amssymb}
\usepackage{pxfonts}
\usepackage{graphicx}
\usepackage{eucal}
\usepackage{mathrsfs}
\usepackage{theorem}
\usepackage{pifont}
\usepackage{color}
\usepackage{mathabx}
\usepackage{enumitem}
\usepackage[margin=2cm]{geometry}
\usepackage[backref=page]{hyperref}
\usepackage[normalem]{ulem}

\usepackage{tocloft}

\definecolor{shadecolor}{rgb}{0.8,0.8,0.8}

\usepackage[all]{xy}
\usepackage{tikz} 
\usepackage{tkz-euclide}

\newcommand{\btkz}{\begin{tikzpicture}}
\newcommand{\etkz}{\end{tikzpicture}}

\newcommand{\brk}[1]{\left(#1\right)}          
\newcommand{\BRK}[1]{\left\{#1\right\}}        

\newcommand{\mymat}[1]{\begin{pmatrix} #1 \end{pmatrix}}
\newcommand{\Cases}[1]{\begin{cases} #1 \end{cases}}

\newcommand{\secref}[1]{Section~\ref{#1}}
\newcommand{\figref}[1]{Figure~\ref{#1}}
\newcommand{\thmref}[1]{Theorem~\ref{#1}}

\newcommand{\propref}[1]{Proposition~\ref{#1}}
\newcommand{\lemref}[1]{Lemma~\ref{#1}}
\newcommand{\corrref}[1]{Corollary~\ref{#1}}

\newcommand{\beq}{\begin{equation}}
\newcommand{\eeq}{\end{equation}}
\newcommand{\bsplit}{\begin{split}}
\newcommand{\esplit}{\end{split}}
\newcommand{\baligned}{\begin{aligned}}
\newcommand{\ealigned}{\end{aligned}}

\newcommand{\Emph}[1]{{\slshape\bfseries #1}}  

\providecommand{\e}{\varepsilon}
\providecommand{\half}{\frac{1}{2}}

\providecommand{\R}{\bbR}

\newcommand{\Textand}{\qquad\text{ and }\qquad}

\providecommand{\vp}{\varphi}

\newcommand{\dist}{{\operatorname{dist}}}

\newcommand{\SO}{{\operatorname{SO}}}

\newcommand{\Vol}{{\operatorname{Vol}}}

\theoremheaderfont{\bfseries\rmfamily}
\newtheorem{theorem}{Theorem}[section]
\newtheorem{lemma}[theorem]{Lemma}
\newtheorem{proposition}[theorem]{Proposition}
\newtheorem{corollary}[theorem]{Corollary}

\newenvironment{proof}{{\flushleft \emph{Proof}:}}{\hfill\ding{110}}
\newenvironment{proof1}[1]{{\flushleft \emph{Proof #1}}}{\hfill\ding{110}}

\def\Xint#1{\mathchoice
   {\XXint\displaystyle\textstyle{#1}}%
   {\XXint\textstyle\scriptstyle{#1}}%
   {\XXint\scriptstyle\scriptscriptstyle{#1}}%
   {\XXint\scriptscriptstyle\scriptscriptstyle{#1}}%
   \!\int}
\def\XXint#1#2#3{{\setbox0=\hbox{$#1{#2#3}{\int}$}
     \vcenter{\hbox{$#2#3$}}\kern-.5\wd0}}

\def\dashint{\Xint-}

\usepackage{framed}
\usepackage{enumitem}
\definecolor{shadecolor}{rgb}{0.90,0.90,0.90}

\newcommand{\calK}{{\mathcal K}}

\newcommand{\calM}{{\mathcal M}}

\newcommand{\calP}{{\mathcal P}}

\newcommand{\calS}{{\mathcal S}}

\newcommand{\bbR}{{\mathbb R}}
\newcommand{\bbS}{{\mathbb S}}

\newcommand{\frakn}{\mathfrak{n}}

\newcommand{\M}{\calM}

\renewcommand{\S}{\calS}
\renewcommand{\P}{\calP}

\newcommand{\g}{g}
\newcommand{\K}{\calK}
\renewcommand{\O}{\operatorname{O}}
\renewcommand{\v}{\mathbf{v}}
\newcommand{\w}{\mathbf{w}}
\newcommand{\W}{\Omega}

\newcommand{\tr}{\operatorname{tr}}

\newcommand{\ip}[1]{\langle #1 \rangle}
\newcommand{\euc}{\mathfrak{e}}

\renewcommand{\Emph}[1]{{\bfseries #1}}

\newcommand{\Gdisc}{{G^\odot_\alpha}}
\newcommand{\gdisc}{{g^\odot_\alpha}}
\newcommand{\Pdisc}{{\calP^\odot_\alpha}}
\newcommand{\Mdisc}{{\M^\odot_h}}
\newcommand{\Sdisc}{{\S^\odot}}
\newcommand{\Sdiscrho}{{\S^\odot_\rho}}
\newcommand{\Sdisch}{{\S^\odot_h}}
\newcommand{\Edisc}{E^\odot_{\alpha,h}}  
\newcommand{\Kdisc}{{\K^\odot_\alpha}}
\newcommand{\ca}{c_\alpha}

\newcommand{\Mdislrho}{{\M^\obot_{\rho,h}}}
\newcommand{\Mdisl}{{\M^\obot_{\e,h}}}
\newcommand{\Sdisl}{{\S^\obot_{\e}}}
\newcommand{\Pdisl}{{\calP^\obot_\e}}
\newcommand{\Kdisl}{{\K^\obot_\e}}
\newcommand{\Gdisl}{{G^\obot_\e}}
\newcommand{\gdisl}{{g^\obot_\e}}
\newcommand{\Edisl}{{E^\obot_{\e,h}}}

\newcommand{\VolE}{\textup{d}\operatorname{Vol}_\euc}
\newcommand{\VolGe}{\textup{dVol}_{\Pdisl}}
\newcommand{\VolGa}{\textup{dVol}_{\Pdisc}}

\newcommand{\dVol}[1]{\textup{dVol}_{ #1}}

\newcommand{\HA}{A}

\newcommand{\EPlate}{\mathcal{E}^h}
\newcommand{\EB}{\mathcal{E}^B}
\newcommand{\ES}{\mathcal{E}^S}
\newcommand{\tEP}{\tilde{\mathcal{E}}^h}
\newcommand{\tEB}{\tilde{\mathcal{E}}^B}

\newcommand{\pl}{\partial}

\numberwithin{equation}{section}

\begin{document}

\title{Energy scaling laws for thin elastic sheets with topological defects}
\author{
Raz Kupferman\footnotemark[1] \and Cy Maor\thanks{Einstein Institute of Mathematics, Hebrew University of Jerusalem} \and David Padilla-Garza\thanks{IST Austria} 
}
\date{}
\maketitle

\begin{abstract}
We derive energy scaling laws for thin elastic sheets with topological defects --- disclinations and dislocations --- for a fully nonlinear 3D model.
For disclinations, the scaling laws are tight in the thickness parameter, and improve upon previous results by applying simultaneously to positive and negative disclinations (e-cones) and by giving an explicit dependence on the defect parameter; this latter dependence is, however, still not tight.
For thin bodies with dislocations, these are, to the best of our knowledge, the first rigorous bounds for models of finite thickness, are tight when the Burgers vector is not large with respect to the thickness, and relate to a well-known conjecture from the physics literature about the scaling.
A main tool is modeling these bodies in the framework of non-Euclidean elasticity, as bodies with a curl-free pre-strain; the curl-freeness allows us to obtain geometric rigidity estimates for the lower bounds.
\end{abstract}

\setcounter{tocdepth}{1}
\begingroup
\footnotesize
\tableofcontents
\endgroup

\section{Introduction and main results}
The study of material defects has been one of the central themes in solid mechanics for more than a century, dating  back to Volterra \cite{Vol07}.
Two of the most important kinds of defects are disclinations and dislocations:
A \textbf{disclination} is a body from which a cylindrical sector has been removed (positive disclinations, or cones) or added (negative disclinations, or $e$-cones).
A \textbf{dislocation} can be obtained by cutting a body along a half-plane, translating one of the sides of the cut and gluing it to the other edge. If the translation is perpendicular to the axis of symmetry, this is called an \emph{edge-dislocation}; throughout this paper, all the dislocations will be of the edge type.
See \figref{fig:defects} for a depiction of disclinations and dislocations.

\begin{figure}[h]
\begin{center}
\includegraphics[height=0.8in]{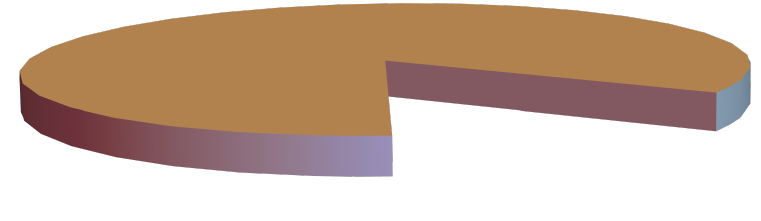}
\hspace{1cm}
\includegraphics[height=1.0in]{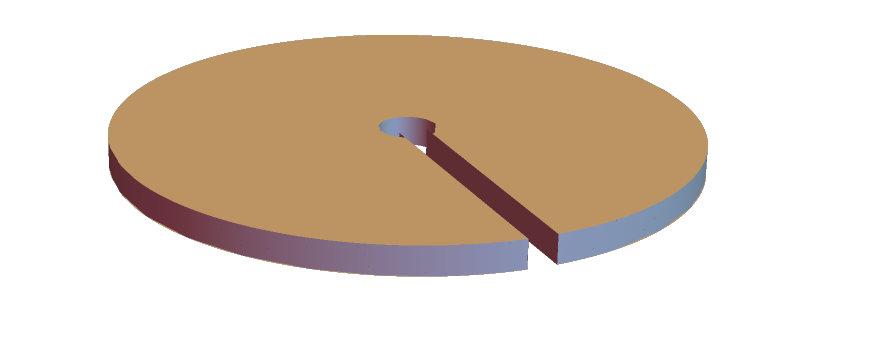}
\end{center}
\caption{Left: a thin sheet with a positive disclination is obtained by gluing the edges of the removed wedge.
	Right: a thin sheet with a dislocation can be obtained by gluing the two parallel edges.}
\label{fig:defects}
\end{figure}

There is a vast literature on the modeling, mechanics and geometry of these bodies; in this paper we are interested in a specific aspect, which is a study of the elastic energy of \emph{thin} bodies with disclinations and dislocations.
Thin bodies with defects are important in many physical systems, e.g., in graphene monolayers.
They have been analyzed in the physics and mechanics literature \cite{ES51,nelson1987fluctuations,SN88,Nel02,Wit07}  using ansatzes, asymptotic and numerical analysis, yet the rigorous mathematical study of them is rather  restricted: The works \cite{MO14,olbermann2016energy,conti2017symmetry,Olb17,olbermann2018shape} study the energy scaling of a single positive disclination in various models (and in the case of \cite{conti2017symmetry}, under external compression), and the works \cite{Kup17,KMP25} study the energy scaling of disclinations and dislocations in the zero-thickness limit (pure bending models).

\subsection{The model}

Bodies with disclinations and dislocations are frustrated --- they exhibit stress even in the absence of external loading --- and as such, they can be naturally modeled within the framework of \emph{non-Euclidean elasticity} (also known as \emph{incompatible elasticity}). 
For a general model, we refer to \cite{Mao25} or to \cite[Sec.~2]{kupferman2026volterra}. We now present the minimal model that fits our purposes:
The elastic body is a three-dimensional domain $\M\subset \R^3$, endowed with a \Emph{prestrain map} $\calP:T\M\to \R^3$, which is a non-degenerate section; we will assume that it is smooth except possibly on a set of measure zero.\footnote{In \cite[Sec.~2]{kupferman2026volterra}, $\calP$ is assumed to be globally-defined and smooth, which is  too strict an assumption for our case. 
In \cite{Mao25} a more general case is treated, in which one covers $\M$ with subdomains endowed with prestrain maps (satisfying certain compatibility conditions). 
In our case, a single prestrain map, which may be discontinuous, suffices.}
The prestrain prescribes the intrinsic geometry of the body: a frame in $T\M$ is oriented and orthonormal if $\calP$ maps it into an oriented orthonormal frame in $\R^3$.
In other words, $\calP$ induces a Riemannian metric $G$ on $\M$ via
\beq\label{eq:metric}
G(u,v) = \ip{\calP(u),\calP(v)}_{\R^3},
\eeq
which we will assume to be smooth.\footnote{This does not contradict the possible non-smoothness of $\calP$ as $G$ is invariant under left-multiplication of $\calP$ by arbitrary $\SO(3)$ fields.}
The elastic energy of a configuration $f:\M\to \R^3$ is then
\beq
\label{eq:non-Euc_energy}
E_{(\M,\calP)}(f) = \dashint_\M \dist^2(df \circ \calP^{-1} , \SO(3) ) \,\dVol{\calP},
\eeq
where $\dVol{\calP}$ is the volume form induced by $\calP$, that is $\calP^*dx$, where $dx$ is the standard volume form on $\R^3$, and $\dashint$ is the integral divided by the total volume.\footnote{One can consider more general energy densities; in fact, for models arising from crystalline materials, non-isotropic models are more physically-relevant. Yet, for the sake of the energy scaling considered in this paper, this simple prototypical energy density suffices.}

We now describe the specific models of disclinations and dislocations considered in this paper.
The key fact is that, as the geometry of both disclinations and dislocations is locally-Euclidean (zero curvature), the prestrain-map $\calP$ can be chosen to be curl-free on its smooth part, that is, $d\calP=0$, when $\calP$ is viewed as an $\R^3$-valued one-form \cite{Mao25}.
This approach is natural both from a modeling perspective, as the map $\calP$ represents the lattice directions at each point, and from an analytic perspective, as it yields Friesecke--James--M\"uller-type (FJM) geometric rigidity estimates without error terms, as will be described below. 
This approach was previously used for the analysis of planar configurations of dislocations \cite{kupferman2026volterra}; here we extend it to both dislocations and disclinations in three-dimensional configurations.\footnote{In most of the literature on non-Euclidean elasticity, $\calP$ is taken to be $\sqrt{G}$ (in the given coordinates on $\M$).
This choice yields the same energy as any other $\calP$ satisfying \eqref{eq:metric} if the energy density is isotropic (as in \eqref{eq:non-Euc_energy}); generally, however, this choice is not curl-free, which makes the analysis significantly harder, as it obscures the flatness of $G$.
An alternative approach in the locally-flat case is to use cut-and-glue methods, e.g., to identify a cone with a ``pac-man shaped" domain in which two edges are glued, as done in \cite{Olb17} for disclinations; here too, the use of a curl-free prestrain is more flexible and simplifies the analysis.}

\paragraph{Disclinations.}
The body manifold of a disclination is a cylinder of radius $R$ and thickness $t$ from which a cylindrical sector of 
angle $\alpha\in (-\infty,2\pi)$ has been removed; in the case of $\alpha<0$, this entails inserting a sector. 
A manifold endowed with such a geometry can be realized as follows: the coordinate representation of the body manifold is the cylinder
\[
N^t_R := \{ (r,\vp,z) ~:~ r\in (0,R),\, \vp\in \bbS^1,\, z\in (0,t)\}.
\]
Take the sector
\[
(r,\theta,z)\in (0,R)\times (0,2\pi-\alpha) \times (0,t),
\]
and endow it
with the standard Euclidean prestrain map (in cylindrical coordinates)
\[
e_1\otimes dx + e_2\otimes dy + e_3\otimes dz = e_1\otimes (\cos\theta \, dr - r\,\sin\theta\,d\theta) + e_2\otimes (\sin\theta \, dr + r\,\cos\theta\,d\theta) + e_3\otimes dz,
\]
where $\{e_1,e_2,e_3\}$ denotes the standard Euclidean frame.
Then, pull back this prestrain map onto $N^t_R$
via the map $\vp = \theta/c_\alpha$, where $c_\alpha = 1 - \alpha/2\pi$.
This results in the prestrain map
\beq
\begin{aligned}
\Pdisc &= e_1\otimes \brk{\cos(\ca\vp) \, dr - \ca r\,\sin(\ca\vp)\,d\vp} \\
&+ e_2\otimes \brk{\sin(\ca\vp) \, dr + \ca r\,\cos(\ca\vp)\,d\vp} \\
&+ e_3\otimes dz.
\end{aligned}
\label{eq:tildeDa}
\eeq
This pullback is equivalent to ``gluing together" the edges of the sector.
As it is more convenient to work with non-dimensionalized parameters, we rescale the system and use the body manifold $\Mdisc :=N^h_1$, where $h=t/R\in (0,1)$ is the non-dimensionalized thickness.
Throughout this work, the superscript $\odot$ denotes entities pertinent to disclinations.

This prestrain map is (generally) discontinuous at $\vp = 0$, yet, it is curl-free on the set of full measure $\Mdisc \setminus \{\vp=0\}$ as a pullback of a curl-free map.
It is easy to check that it indeed induces the standard metric of a cylindrical cone
\beq
\Gdisc = dr\otimes dr + \ca^2 r^2\,d\vp\otimes d\vp + dz\otimes dz.
\label{eq:Gdisc}
\eeq

The case of $\alpha>0$ represents standard cones, i.e., cones of positive curvature, and the case of $\alpha<0$ represents $e$-cones, i.e., cones of negative curvature. 
To avoid technicalities, and to focus on more physically-relevant scenarios, we restrict ourselves to $\alpha\in (-\pi,\pi)$. 
We shall also assume that the radius of the body is not small with respect to its thickness, that is, $h<1/2$.

We denote by 
\[
\Edisc = E_{(\Mdisc,\Pdisc)}
\]
the energy associated with this disclination model.

\paragraph{Dislocations.}
The body manifold of a dislocation is an annulus
\[
\M_{r_0,R}^t := \{ (r,\vp,z) ~:~ r \in (r_0, R),\, \vp\in \bbS^1,\, z\in (0,t)\},
\]
which we again non-dimensionalize, and consider $\Mdislrho :=  \M^h_{\rho,1}$, where $\rho = r_0/R$ is the core radius and $h = t/R$ is the thickness; the superscript $\obot$ denotes henceforth entities pertinent to dislocations.

The prestrain map of a dislocation with a Burgers vector of size $\e$ can be chosen to be 
\beq
\label{eq:implantdisl}
\Pdisl  = I_{3\times 3} +  \frac{\e}{2\pi} e_1\otimes d\vp,
\eeq
where $I_{3\times 3}$ is the differential of the coordinate identity map; indeed, $\Pdisl$ is curl-free, however its integral over a simple path encircling the origin yields a vector of size $\e$ (see \cite[Sec.~3]{kupferman2026volterra} for a detailed analysis of this geometry, as well as for a proof that it is the unique geometry of an edge-dislocation).
Note that in order for $\Pdisl$ to be invertible, $\e$ cannot be large with respect to the core radius $\rho$. 
In most physical systems, the core radius and the Burgers vector are of the same order (that of the crystal lattice spacing), and thus we will henceforth assume that $\rho=\e$, i.e., we do not artificially enlarge the core.
As for disclinations, we shall also assume that the radius of the body is not small with respect to its thickness, that is, $h,\e<1/2$.

We denote by 
\[
\Edisl = E_{(\Mdisl,\Pdisl)}
\]
the energy associated with this dislocation model, and by $\Gdisl$ the metric induced by $\Pdisl$.

\paragraph{Two-dimensional models}
Thin sheets can be viewed as ``thickened surfaces", and their energetics is tightly related to the energetics 
associated with the immersion of surfaces in $\R^3$. The disclination and the dislocation models involve geometries having axial symmetry, i.e., independent of $z$ (a \emph{plate} geometry in the mechanical context). 
Denoting by $\S\subset\M$ the $\{z=0\}$ surface, henceforth referred to as the \emph{mid-surface}, the prestrain $\P$ and the metric $G$ decompose orthogonally, yielding a metric $g$ and a prestrain $\K$ on $\S$; see Appendix~\ref{sec:2D3D} for details. 

Let $\sigma:\S\to\R^3$ be a configuration of the mid-surface. The plate energy associated with $\sigma$ is of the form
\[
\EPlate_{(\S,\K)}(\sigma) = \ES_{(\S,K)}(\sigma) + h^2\, \EB_{\S}(\sigma),
\]
where $\ES_{(\S,K)}(\sigma)$ is the so-called \emph{stretching energy} of the surface, which penalizes distortions with respect to the metric $g$, and $ \EB_{\S}(\sigma)$ is the so-called \emph{bending energy} of the surface, which penalizes its extrinsic curvature. The 3D energy $E_{(\M,\P)}(f)$ and the plate energy $\EPlate_{(\S,\K)}(\sigma)$ can be related either by viewing $\sigma$ as the restriction of $f$ to the mid-surface, or by viewing $f$ as an extension of $\sigma$ via the so-called Kirchhoff-Love extension \cite{Lov27}. Furthermore, two-sided inequalities can be derived, connecting the reduced 2D and full 3D models. This schematic description of 2D versus 3D models suffices for the sake of this introduction;  further refinements and technicalities are provided in Appendix~\ref{sec:2D3D}.

\subsection{Main results}

\paragraph{Disclinations.}
We prove the following bounds for disclinations:
\begin{theorem}[Energy scaling law for thin sheets with disclinations]
\label{thm:disclination}
Let $h<1/2$ and $\alpha \in (- \pi, \pi)$, and let $\Edisc$ be the elastic energy of a body with a disclination as defined in the previous section. 
Then, 
\begin{equation}
\frac{1}{C} \alpha^2 h^2 \,\log\brk{\frac1h} \le \inf_{f\in W^{1,2}} \Edisc(f) \leq C \min \BRK{\alpha^2,  |\alpha| h^2 \log\brk{\frac1h}} , 
\end{equation}
for a constant $C>0$ independent of $h$ and $\alpha$.  
\end{theorem}

We discuss the proof of the lower bounds below (as it is similar to the dislocation case); the upper bound consists of the minimum of two terms:
\begin{itemize}
\item The $\alpha^2$ bound corresponds to a flat configuration of the mid-surface --- i.e., zero bending energy. It is optimal when $h$ is large.
\item The $|\alpha| h^2 \log\brk{1/h}$ bound corresponds to an isometric immersion of the mid-surface (for $\alpha>0$, this is the standard cone), outside a small region of radius $h$ around the origin --- i.e., zero stretching energy along the mid-surface throughout most of the domain.
\end{itemize}
This bound is tight in the scaling of $h$ as $h\to 0$, yet it is not tight in the scaling of $\alpha$ for $\alpha\to 0$.
Nevertheless, it constitutes a first bound for the full 3D model of disclinations in thin sheets (rather than reduced 2D models) that has explicit dependence on $\alpha$; indeed, in \cite[Theorem 2]{Olb17} the constant depends on $\alpha$.
It is also the first bound to apply both for positive and negative $\alpha$ --- the sequence of works \cite{MO14,olbermann2016energy,Olb17,olbermann2018shape} only deals with $\alpha>0$, whereas \cite{KMP25} deals with arbitrary $\alpha$ but only for pure bending models in the limit $h\to 0$.

We believe the upper bound to be tight.
Indeed, a lower bound of $ |\alpha| h^2 \log\brk{1/h}$ was obtained for a reduced 2D model with a geometrically-linearized bending \cite{Olb17,olbermann2018shape}, in a von-K\'arm\'an model \cite{Olb17} and in a pure bending model \cite{KMP25}.
To prove its tightness, it would be sufficient to obtain an $ |\alpha| h^2 \log\brk{1/h}$ lower bound for a 2D model, since, as shown in \thmref{thm:3D2D}, such a bound would also apply to the 3D setting.
Yet upgrading the pure bending bound of \cite{KMP25} to allow for stretching is significantly more complicated than in the case of geometrically-linearized bending \cite{Olb17,olbermann2018shape}, and we aim to address it in future work.
We note, however, that the current $ \alpha^2 h^2 \,\log\brk{1/h}$ lower bound, while presumably non-optimal, is both more informative than previous bounds for the 3D model, and has a much simpler proof.

\paragraph{Dislocations.}
We prove the following bounds for dislocations:
\begin{theorem}[Energy scaling law for thin sheets with edge-dislocations]
\label{thm:dislocation}
Let $\e,h<1/2$, and let $\Edisl$ be the elastic energy of a body with a dislocation as defined in the previous section. 
Then, 
\begin{equation}\label{eq:disloc_lb}
\frac{1}{C} \min\left\{\e^2, h^2\right\} \log\brk{2+\frac{h}{\e}} \le \inf_{f\in W^{1,2}} \Edisl(f) \leq C \min \left\{\e^2\log\brk{\frac{1}{\e}}, h^2 + \e^2 , h^2 \log\brk{\frac{1}{\e}}\right\} , 
\end{equation}
for a constant $C>0$ independent of $h$ and $\e$.  
\end{theorem}

The lower bound is obtained by summing up lower bounds for the elastic energy of annuli of width $h$.
The upper bound consists of the minimum of three terms:
\begin{itemize}
\item The $\e^2\log(1/\e)$ bound corresponds to a flat configuration of the mid-surface (zero bending). It is optimal when $h$ is large.
\item The $h^2 \log(1/\e)$ bound corresponds to an isometric immersion of the mid-surface (zero stretching). It is optimal when $h$ is very small. 
\item The $h^2 + \e^2$ bound corresponds to a single-fold construction due to Conti \cite{Con23}, that has similar bending and stretching content in the mid-surface. It is optimal when $h$ and $\e$ are of the same order.
\end{itemize}
Note that if we assume that $\e = h^\beta$ for some $\beta\ge 0$, then the bound is tight for $\beta\ge 1$, and scales like $\e^2\log(1/\e)$ for $\beta>1$ and like $\e^2 = h^2$ when $\beta = 1$.
For $\beta<1$, however, we have a lower bound of $h^2$ and an upper bound of $h^2\log(1/\e)$. 
We believe that the upper bound is tight, i.e., that in this case there is a logarithmic divergence on top of the $h^2$ scaling.
Indeed, this is the case in the iterated limit of taking $h\to 0$ while fixing $\e$, and then taking $\e\to 0$. 
Specifically, the $\Gamma$-limit of $h^{-2}\Edisl$ is a Kirchhoff bending model \cite{LP11,KS14}, which was proved to diverge like $\log(1/\e)$ \cite{Kup17,KMP25}.
Note that the proof of this divergence is based on propagation of the bending energy from the inner boundary outwards, and thus an approach based on summation of independent contributions of different annuli, like in our lower bound (see below), cannot be expected to bridge this gap; a new idea or approach is needed.

Finally, we note the relation of our results to a well-known conjecture in the physics literature by Nelson and Peliti \cite{nelson1987fluctuations}. 
They conjectured that the energy of a thin body with a dislocation (divided by $h^2$) does not diverge as $\e,h\to 0$; to be more precise, their conjecture concerns the case where the mid-surface is isometrically immersed (zero stretching).
Seung and Nelson \cite{SN88} provided numerical evidence for that, but their model does allow for some stretching.
The lower bound of $\log(1/\e)$ for isometric immersions obtained in \cite{Kup17,KMP25} refutes the original (zero stretching) Nelson--Peliti conjecture, yet the upper bound \eqref{eq:disloc_lb} implies that it is valid when one allows for stretching, at least when $\e\sim h^\beta$ for $\beta\ge 1$. 

\paragraph{Proof strategy.}
For the lower bounds, a key observation is that if the prestrain map $\calP$ of an elastic body is curl-free, then one is able to obtain an FJM-type geometric rigidity estimate without error terms: 
That is, for every configuration $f\in W^{1,2}(\M;\R^3)$ there exists $Q\in \SO(3)$ such that
\beq
\label{eq:rigidity}
\int_{\M}   |df - Q \calP|^2 \, \dVol{\calP} 
\le 
C \, \int_{\M} \dist^2(df\circ \calP^{-1},\SO(3))\,  \dVol{\calP}  
\eeq
for some $C>0$ depending only on $\M$ and $\calP$.
This should be compared to the case of a general $\calP$, in which one has an additional error term on the right-hand side (see, e.g., \cite[Theorem 2.3]{LP11}, \cite[Sec.~2.2.1]{MS19}).
Indeed, if $\calP$ is curl-free, then the induced metric is locally-Euclidean \cite[Prop.~3.2]{Mao25}, and thus $\M$ can be covered by subdomains which are isometric to domains in $\R^3$, for which we can use the standard rigidity estimate \cite{FJM02b}.
The constant $C$ then depends on the rigidity constants of these subdomains and their intersections. 
This observation was used in the derivation of strain-gradient plasticity as a limit for bodies with many dislocations in the non-Euclidean elasticity model \cite[Theorem~4.4]{kupferman2026volterra} (a similar idea was used before, in a slightly different context, in \cite[Prop.~3.3]{SZ12}). 
The lower bounds in this work are based on obtaining these estimates, controlling the relevant rigidity constants in terms of $h,\e,\alpha$, and estimating the left-hand side using the fact that $\calP$ is not exact, i.e., not a derivative.

The upper bounds are based on explicit ansatzes for the mid-surface, estimates of the stretching and bending of these configurations, and the relations between 2D models and the full 3D energy as described above and detailed in Appendix~\ref{sec:2D3D}.

\paragraph{Notations}
Incompatible elasticity represents bodies as Riemannian manifolds $(\M,G)$. A metric $G$ induces norms, denoted by $|\cdot|_G$, on vector fields, forms, and their tensor products. The norm of a bundle map between vector bundles endowed with metrics $G_1$ and $G_2$ is denoted by $|\cdot|_{G_1,G_2}$. The $L^p$ norm of sections is always taken with respect to these pointwise norms and the volume form. 
Throughout this work, we use the symbol $G$ to denote metrics of three-dimensional manifolds and the symbol $g$ to denote metrics of two-dimensional manifolds.
We use the generic notation $C$ for a positive constant that does not depend on specified parameters;  its value may change from one line to another. 

\paragraph{Structure of the paper.}
In \secref{sect:lowbound} we derive the rigidity estimates of the type \eqref{eq:rigidity} and obtain the lower bounds in Theorems~\ref{thm:disclination}--\ref{thm:dislocation}.
In \secref{sect:upbounds} we prove the upper bounds in Theorems~\ref{thm:disclination}--\ref{thm:dislocation} by constructing the various ansatzes.
In Appendix~\ref{sec:2D3D} we prove the relations between reduced 2D models and the full 3D model.

\paragraph{Acknowledgements.}
We are grateful to Sergio Conti for sharing with us his construction of the ansatz for embedding a thin body with a dislocation.
RK was partially supported by ISF Grant 560/22,  CM was partially supported by ISF grant 2304/24 and BSF grant 2022076, and DPG was supported by the Zuckerman STEM Leadership Program. 
Part of this work was written while CM was visiting the University of Toronto and the Fields Institute; CM is grateful for their hospitality.

\section{Lower bounds}
\label{sect:lowbound}

In this section, we prove the parts of \thmref{thm:disclination} (disclinations) and \thmref{thm:dislocation} (dislocations) concerning the lower bounds.

\subsection{{Dislocations}}

We start by noting the following bounds concerning the geometry of a dislocation, induced by the prestrain map $\Pdisl$:

\begin{proposition}
\label{prop:prelim}
Let $\Pdisl$ be given by \eqref{eq:implantdisl}, {with} $\Gdisl$ the corresponding metric. {Let} $\iota: \Mdisl\to\R^3$ denote the configuration corresponding to the coordinate inclusion map,
namely,
\[
\iota(r,\vp,z) = (r\,\cos\vp,r\,\sin\vp,z), 
\]
and {let} $\euc$ denote the Euclidean metric in $\R^3$. Then: 
\begin{itemize}
\item[(a)] \[||dr|_{\Gdisl} -1 | \le \frac{\e}{2\pi r - \e}
\Textand
||r\,d\vp|_{\Gdisl} -1 | \le \frac{\e}{2\pi r - \e}.\]

\item[(b)] 
\beq
|\Pdisl - d\iota|_{\Gdisl,\euc} \le  \frac{\e}{r}, \Textand |d\iota^{-1} - \Pdisl^{-1}|_{\euc,\Gdisl} \leq C \frac{\e}{r},
\label{eq:disc_iota_est}
\eeq
for some absolute constant $C$. 
In particular, the metrics $\Gdisl$ are equivalent to the Euclidean metric on the domain $\{r>\e\}$, with the equivalence constant independent of $\e$.
\end{itemize}
\end{proposition}

\begin{proof}
This result is proved in \cite[Sec.~3]{kupferman2026volterra}.
Specifically, (a) is Eq.~(3.9) and (b) is part of Proposition 3.10 in \cite{kupferman2026volterra}. 
\end{proof}

Towards obtaining an FJM-type rigidity estimate for bodies with dislocations, we begin by estimating the geometric rigidity constant (of the  Euclidean FJM rigidity theorem \cite{FJM02b}) of Euclidean cylindrical half-annuli, which will be needed in the proof. 

\begin{lemma}
\label{cor:FJMHA}
For $0<\rho_0<\rho_1$ and $h>0$, denote by
\[
\HA^h_{\rho_0,\rho_1}
=
\{(x,y,z)\in\mathbb R^3:\ x>0,\ \rho_0^2<x^2+y^2<\rho_1^2,\ 0<z<h\}
\]
the cylindrical half-annulus of height $h$ and inner and outer radii $\rho_0,\rho_1$. 
Then, for all $\rho,h>0$, the cylindrical half-annulus $\HA_{\rho,\rho+h}^h$ admits a rigidity constant $C(1+(\rho/h)^2)$, for some absolute constant $C$.
That is, for every $f\in W^{1,2}(\HA_{\rho,\rho+h}^h;\R^3)$ there exists a rotation $Q \in \SO(3)$, such that
\beq
\int_{\HA_{\rho,\rho+h}^h}   |df - Q|^2 \, dx
\le 
 C\brk{1+\frac{\rho^2}{h^2}}\, \int_{\HA_{\rho,\rho+h}^h} \dist^2(df ,\SO(3))\, dx.
\label{eq:FJMHA} 
\eeq
\end{lemma}

\begin{proof}
The FJM rigidity constant  is invariant under scaling \cite{FJM02b}.
Thus, in order to cover all parameter regimes of $\rho$ and $h$ it is sufficient to consider $\HA_{\delta,1+\delta}^1$ and $\HA_{1,1+\delta}^\delta$ , with $\delta\le 1$ corresponding to $\rho/h$ in the first case and $h/\rho$ in the second case. 
Moreover, uniformly bi-Lipschitz domains admit a uniform rigidity constant. 
Since it can be easily shown that the domains $\HA_{\delta,1+\delta}^1$, with $\delta\le 1$, are uniformly bi-Lipschitz, they admit a rigidity constant $C$. 

Similarly, $\HA_{1,1+\delta}^\delta$ for $\delta\in [1/10,1)$ can easily be shown to be bi-Lipschitz equivalent to $\HA_{1,1+1/10}^{1/10}$, whereas $\HA_{1,1+\delta}^\delta$ for $\delta<1/10$ can easily be shown to be uniformly bi-Lipschitz equivalent to rods of length $1$ and a square cross section of side length $\delta$. Rods of these dimensions have a rigidity constant that scales like $1/\delta^2$; one can deduce {this scaling} from a similar argument for plates as in \cite[Theorem~4.8]{Lew23}; {for completeness, we give a different, self-contained proof, similar in spirit to the argument in \cite[Theorem~2.1]{MM03}.}

Let $N:=\left\lceil \frac{1 }{\delta}\right\rceil $, and partition the rod $(0,1)\times (0,\delta)^2$ into $N$ boxes $K_i = (i/N,(i+1)/N)\times (0,\delta)^2$ for $i=0,\ldots,N-1$.

Therefore, there exists $C>0$ such that for each $i$ there exists
$R_i\in SO(3)$ such that
\[
\int_{K_i}|df-R_i|^2\,dx
\leq
C\int_{K_i}\dist^2(df,SO(3))\,dx .
\]
Moreover, for each pair of adjacent boxes $K_i,K_{i+1}$, the union
$K_i\cup K_{i+1}$ is also uniformly bi-Lipschitz equivalent to a cube. Hence there exists $R_{i,i+1}\in SO(3)$ such that
\[
\int_{K_i\cup K_{i+1}} |df-R_{i,i+1}|^2\,dx
\leq
C\int_{K_i\cup K_{i+1}}\dist^2(df,SO(3))\,dx .
\]
Since $|K_i|\simeq \delta^3$, the preceding two estimates imply
\begin{equation}
\label{eq:diffdisc}
|R_i-R_{i+1}|^2
\leq
C \delta^{-3}
\int_{K_i\cup K_{i+1}}\dist^2(df,SO(3))\,dx.
\end{equation}

Let
\[
\overline R:=\frac1N\sum_{i=0}^{N-1} R_i .
\]
By the discrete Poincar\'e inequality on the discrete space $\{0,\ldots,N-1\}$, 
\[
\sum_{i=0}^{N-1}|R_i-\overline R|^2
\leq
C N^2
\sum_{i=0}^{N-2}|R_i-R_{i+1}|^2 .
\]
Let $R\in SO(3)$ be a nearest-point projection of $\overline R$ onto
$SO(3)$. Since each $R_i$ belongs to $SO(3)$,
\[
|\overline R-R|\leq |\overline R-R_i|
\]
for every $i$, and therefore
\[
|R_i-R| \leq |R_i-\overline R| + |\overline R-R| \leq 2|R_i-\overline R|.
\]
It follows that
\[
\sum_{i=0}^{N-1}|R_i-R|^2
\leq
C N^2
\sum_{i=0}^{N-2}|R_i-R_{i+1}|^2 .
\]

We now estimate
\[
\begin{aligned}
\int_{{(0,1)\times (0,\delta)^2}} |df-R|^2\,dx
&\leq
C\sum_{i=0}^{N-1}\int_{K_i}|df-R_i|^2\,dx
+
C\sum_{i=0}^{N-1}|K_i|\,|R_i-R|^2  \\
&\leq
C\int_{{(0,1)\times (0,\delta)^2}}\dist^2(df,SO(3))\,dx
+
C \delta^3 N^2
\sum_{i=0}^{N-2}|R_i-R_{i+1}|^2 .
\end{aligned}
\]
Using the estimate on the differences of adjacent rotations \eqref{eq:diffdisc}, we obtain
\[
\int_{{(0,1)\times (0,\delta)^2}} |df-R|^2\,dx
\leq
C N^2
\int_{{(0,1)\times (0,\delta)^2}} \dist^2(df,SO(3))\,dx .
\]
Since $N\simeq \delta^{-1}$, we are done.
\end{proof}

The above estimate for Euclidean half-annuli yields the following rigidity estimates for dislocations:

\begin{proposition}
\label{prop:FJMdisloc}
Let $\e\le \rho \leq 1-h$ and let $h<1/2$.
Denote by $\M_{\rho,\rho+h}^h \subset\Mdisl$ the annulus obtained by restricting the radial coordinate between $\rho$ and $\rho+h$.
Then there exists for every $f\in W^{1,2}(\M_{\rho,\rho+h}^h;\R^3)$  a rotation $Q \in \SO(3)$, such that
\beq
\label{eq:FJMdisloc}
\int_{\M_{\rho,\rho+h}^h}   |df - Q \Pdisl|_{\Gdisl,\euc}^2 \, \VolGe 
\le 
C \brk{1+\frac{\rho^2}{h^2}}\, \int_{\M_{\rho,\rho+h}^h} \dist^2(df\circ \Pdisl^{-1},\SO(3))\, \VolGe, 
\eeq
where $C>0$ is an absolute constant.
\end{proposition}

\begin{proof}
We proceed as in the proof of \cite[Theorem~4.4]{kupferman2026volterra}, which proves a similar geometric rigidity estimate for 2-dimensional annuli of fixed width. 
First, we cover $\M_{\rho,\rho+h}^h$ with 3 overlapping cylindrical half-annuli, as shown in \figref{fig:covering}, denoted $\Omega_{1}, \Omega_{2}, \Omega_{3}$; note that these are half-annuli with respect to the coordinates, but not with respect to the intrinsic metric induced by the prestrain map $\Pdisl$.

Since the domains $\Omega_i$ are simply-connected and $\Pdisl$ is curl-free, there exist isometric immersions $\Psi_i : \Omega_i \to \R^3$, that is $d\Psi_i = \Pdisl$.
In fact, $\Psi_i$ are isometric embeddings: up to rotations, they are functions of the type $\Psi(r,\vp,z) = r e_r + z e_3 + (\e/2\pi)\vp e_1$, where $e_r = (\cos\vp,\sin\vp,0)$, which is one-to-one on $\{x_1>0\}\cap \{r>\e\}$.
By the distortion estimates \eqref{eq:disc_iota_est}, the domains $\Psi_i(\Omega_i)$ are bi-Lipschitz equivalent to $\Omega_i$, uniformly in $\e,\rho,h$.
By \lemref{cor:FJMHA}, there exists a universal $C>0$ such that for every $\tilde{f}\in W^{1,2}(\Psi_i(\Omega_i);\R^3)$ there exists a rotation $Q_i \in \SO(3)$, such that
\beq
\int_{\Psi_i(\Omega_i)}   |d\tilde{f} - Q_i |_{\euc,\euc}^2 \, dx 
\le 
C \brk{1+\frac{\rho^2}{h^2}}\, \int_{\Psi_i(\Omega_i)} \dist^2(d\tilde{f},\SO(3))\, dx.
\eeq
Changing coordinates back to $\Omega_i$, we get 
\beq
\label{eq:FJMdisloc}
\int_{\Omega_i}   |df - Q_i \Pdisl|_{\Gdisl,\euc}^2 \, \VolGe 
\le 
C \brk{1+\frac{\rho^2}{h^2}}\, \int_{\Omega_i} \dist^2(df\circ \Pdisl^{-1},\SO(3))\, \VolGe. 
\eeq
Since $\Omega_1$ overlaps both $\Omega_2$ and $\Omega_3$, $Q_1$  must be close to both $Q_2$ and $Q_3$, because they approximate the same deformation gradient on the overlap (see the proof of \cite[Theorem~4.4]{kupferman2026volterra}).
Hence, choosing for example $Q=Q_1$, we obtain \eqref{eq:FJMdisloc}. 
\end{proof}

\begin{figure}[htbp]
    \centering 
    \includegraphics[width=0.7\textwidth]{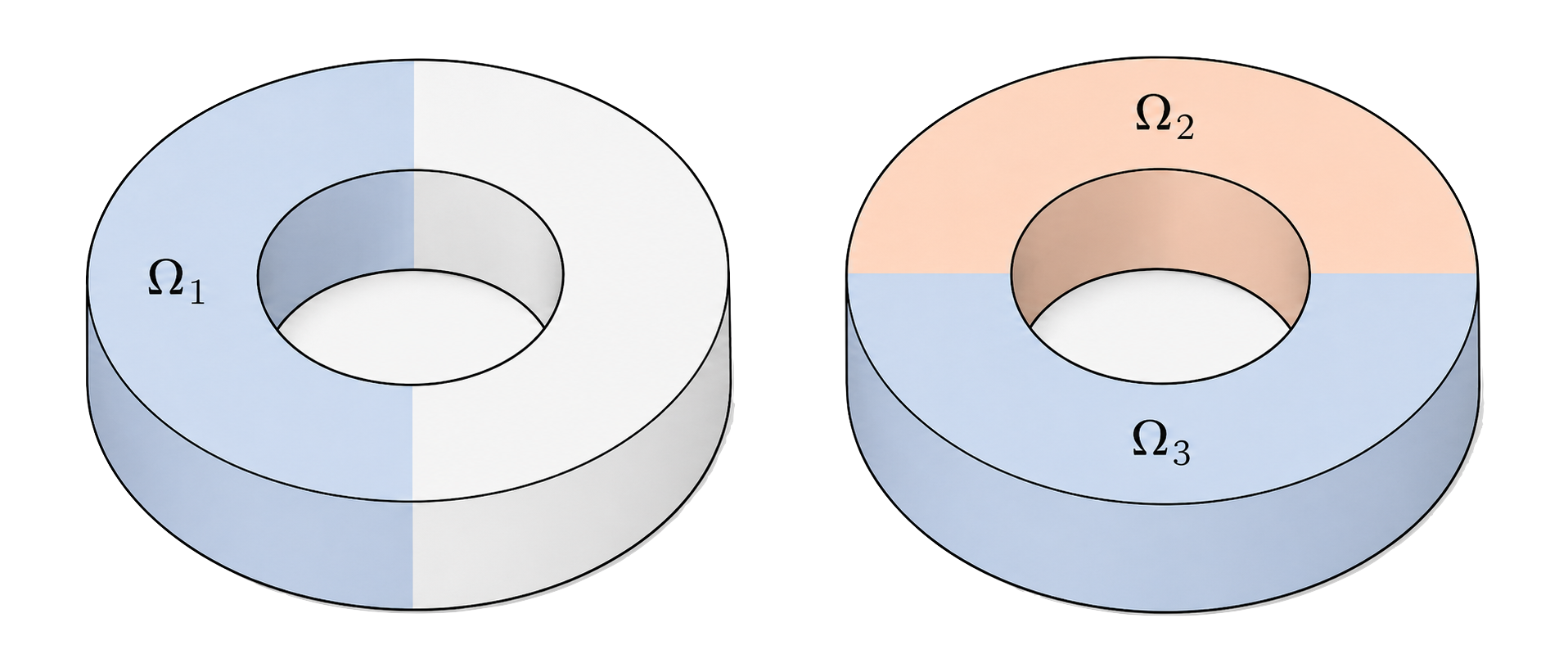} 
    \caption{Depiction of the covering used in the proof of \propref{prop:FJMdisloc}.}
    \label{fig:covering} 
\end{figure}

We  use the previous proposition to derive a lower bound for the elastic energy in a thin ring of a thin body with a dislocation. 

\begin{lemma}
\label{lem:lowbound}
Let $\e\le \rho \leq 1-h$ and let $h<1/2$.
Then for every $f\in W^{1,2}(\M_{\rho,\rho+h}^h;\R^3)$ and every rotation $Q \in \SO(3)$, there holds
\[
\int_{\M_{\rho,\rho+h}^h}   |df - Q \Pdisl|_{\Gdisl,\euc}^2 \, \VolGe \ge C \e^2 h\log\brk{1+\frac{h}{\rho}}
\]
for some absolute constant $C>0$.
\end{lemma}

\begin{proof}
Set 
\[ 
A:=df-Q\Pdisl . 
\]
 
Since $r\geq \rho\geq \e$, the metric $\Gdisl$ is uniformly comparable to the Euclidean metric (\propref{prop:prelim}), and hence
\beq
\int_{M^h_{\rho,\rho+h}} |A|^2_{\Gdisl,\euc}\, d\operatorname{Vol}_{\Pdisl} \geq C\int_0^h\int_\rho^{\rho+h}\int_0^{2\pi} \frac{|A(\partial_\vp)|^2}{r}\, d\vp\,dr\,dz,
\label{eq:lb_lemma_1}
\eeq
for some universal constant $C$.

For a.e.\ pair $(r,z)$, the map $\vp\mapsto f(r,\vp,z)$ belongs to $W^{1,2}(\bbS^1;\mathbb R^3)$, and hence 
\[ 
\int_0^{2\pi} df(\partial_\vp)\,d\vp=  0. 
\] 
On the other hand, 
\[ 
\int_0^{2\pi} \Pdisl(\partial_\vp)\,d\vp = \int_0^{2\pi} d\iota(\partial_\vp)\,d\vp + \frac{\e}{2\pi}e_1 \int_0^{2\pi}
d\vp = \e e_1.
\]
Consequently, 
\[ 
\int_0^{2\pi} A(\partial_\vp)\,d\vp = -\e Q\, e_1 . 
\] 
By Jensen's inequality, 
\[ 
\int_0^{2\pi}|A(\partial_\vp)|^2\,d\vp \geq \frac{1}{2\pi} \left| \int_0^{2\pi}A(\partial_\vp)\,d\vp \right|^2 = \frac{\e^2}{2\pi}. 
\] 
Substituting this estimate into \eqref{eq:lb_lemma_1} gives
\[ 
\int_{M^h_{\rho,\rho+h}} |df-Q\Pdisl|^2_{\Gdisl,e}\, d\operatorname{Vol}_{\Pdisl} \geq 
C\e^2 \int_0^h\int_\rho^{\rho+h}\frac{1}{r}\,dr\,dz  = 
C \e^2 h \log\left(1+\frac{h}{\rho}\right), 
\]
which proves the claim. 
\end{proof}

We now complete  the proof of the lower bound in \thmref{thm:dislocation}:

\begin{proof1}{of lower bound in \thmref{thm:dislocation}:}
Let $f\in W^{1,2}(\Mdisl;\mathbb R^3)$. Since
$\e,h<1/2$, the total volume of $\Mdisl$ with respect to
$\Gdisl$ is comparable to $h$. By the definition of the normalized
energy,
\[
\Edisl(f)
\geq
\frac{C}{h}
\int_{\Mdisl}
\dist^2(df\circ (\Pdisl)^{-1},SO(3))\,d\Vol_{\Pdisl}.
\]
We now partition $\Mdisl$ into annuli
\[
\M_{r_i,r_i+h}^h, \qquad r_i = \e + ih , \qquad i = 0,\ldots ,N,
\]
where $N = \lfloor \frac{1-\e}{h}\rfloor-1$,
and bound from below the right-hand side in each subdomain separately.

\textbf{Step 1: The inner-most annulus.}
We first estimate the contribution of the inner-most annulus
$\M^h_{\e,\e+h}$. By \propref{prop:FJMdisloc}
with $\rho=\e$, along with \lemref{lem:lowbound}, there exists a $Q\in SO(3)$ such that
\[
\begin{aligned}
\int_{M^h_{\e,\e+h}}
\dist^2(df\circ (\Pdisl)^{-1},SO(3))\,d\Vol_{\Pdisl}
&\geq
\frac{C}{1+\e^2/h^2}
\int_{\M^h_{\e,\e+h}}
|df-Q\Pdisl|^2_{\Gdisl,e}\,d\Vol_{\Pdisl} \\
&\geq
\frac{C\,\e^2 h}{1+\e^2/h^2}
\log\left(1+\frac{h}{\e}\right).
\end{aligned}
\]
Normalizing by the total volume of the body, the contribution of the inner-most annulus to the energy $\Edisl(f)$ is bounded from below by
\begin{equation}
\frac{C\e^2}{1+\e^2/h^2}
\log\left(1+\frac{h}{\e}\right).
\label{eq:inner-annulus-dislocation-lower-bound}
\end{equation}

\textbf{Step 2: The outer annuli.}
We next estimate the contribution of the remaining annuli. 
Assume that either $\e<1/10$ or $h<1/10$, since if both $\e,h\in [1/10,1/2)$ then the inner annulus already gives the wanted lower bound. 
Applying \propref{prop:FJMdisloc} and \lemref{lem:lowbound} on each annulus
$\M^h_{r_i,r_i+h}$, and using that $r_i>h$ for $i\ge 1$, we obtain
\[
\int_{\M^h_{r_i,r_i+h}}
\dist^2(df\circ (\Pdisl)^{-1},SO(3))\,d\Vol_{\Pdisl}
\geq
\frac{C\,\e^2 h}{1+r_i^2/h^2} \log\left(1+\frac{h}{r_i}\right)
\ge C\,\e^2 h^4\,\frac{1}{r_i^3},
\]
where in the last passage we used the fact that $\log(1+x) \ge x/2$ for $x\in(0,1)$.

Summing over $i$ and normalizing, we get that the total energetic contributions of the outer annuli can be bounded from below by
\beq
\label{eq:outer-annuli-dislocation-lower-bound}
\frac{C}{h}
\sum_{i=1}^{N}
\e^2 h^4\,\frac{1}{r_i^3}
=
C\,\e^2 h^3
\sum_{i=1}^N \frac{1}{r_i^3}
\ge c\e^2h^2 \int_{\e+h}^{1}\frac{dr}{r^3}
\ge 
C\,\e^2 h^2 \left(\frac{1}{(\e+h)^2}-1\right)
\geq 
C\min\{\e^2,h^2\},
\eeq
where in the last inequality we used the fact that $\e+h<3/5$, so that $1/(\e+h)^2-1 \ge 2/5(\e+h)^2$.

\textbf{Step 3: Conclusion.}
Combining the estimates
\eqref{eq:inner-annulus-dislocation-lower-bound}--\eqref{eq:outer-annuli-dislocation-lower-bound}, we obtain
\[
\Edisl(f)
\geq
C\left[
\frac{\e^2}{1+\e^2/h^2}
\log\left(1+\frac{h}{\e}\right)
+
\min\{\e^2,h^2\}
\right].
\]

It remains only to compare the right-hand side with the desired expression.
If $\e\leq h$, then
\[
\frac{\e^2}{1+\e^2/h^2}
\log\left(1+\frac{h}{\e}\right)
\geq
C\e^2
\log\left(2+\frac{h}{\e}\right) \ge C\e^2 = C \min\{\e^2,h^2\}.
\]
If $\e\geq h$, then
\[
\min\{\e^2,h^2\}=h^2
\geq
C h^2 \log\left(2+\frac{h}{\e}\right).
\]
Hence, in all cases,
\[
\Edisl(f)
\geq
C\min\{\e^2,h^2\}
\log\left(2+\frac{h}{\e}\right).
\]
Taking the infimum over $f$ proves the lower bound.
\end{proof1}

\subsection{Disclinations}
\label{subsect:disclower}

In this section we prove the part in \thmref{thm:disclination} concerning the lower bound for disclinations.
Again, we start with a rigidity estimate, which is the analogue of \propref{prop:FJMdisloc}. 

\begin{proposition}
\label{prop:FJMdisc}
Let $h<1$, $\alpha \in (- \pi, \pi)$ and $\rho\in [h,1-h]$.
Then there exists an absolute constant $C>0$ such that for every $f\in W^{1,2}(\M_{\rho,\rho+h}^h;\R^3)$ there exists a rotation $Q \in \SO(3)$, such that
\[
\int_{\M_{\rho,\rho+h}^h}   |df - Q \Pdisc|_{\Gdisc,\euc}^2 \, \VolGa 
\le 
C \frac{\rho^2}{h^2}\, \int_{\M_{\rho,\rho+h}^h} \dist^2(df\circ \Pdisc^{-1},\SO(3))\, \VolGa.  
\]
\end{proposition}

\begin{proof}
The proof is as in \propref{prop:FJMdisloc}, using the fact that $\Pdisc$ is a parallel section of $\SO(\Gdisc,\euc)$ on $\M_{\rho,\rho+h}^h\setminus\{\vp=0\}$. 
Indeed, covering $\M_{\rho,\rho+h}^h$ with $\Omega_1,\Omega_2,\Omega_3$ as in \propref{prop:FJMdisloc}, and noting that $\Pdisc$ is smooth and closed on each $\Omega_i$, 
there are immersions $\Psi_i:\Omega_i\to \R^3$ with $d\Psi_i = \Pdisc$,  which are thus isometric immersions of $\Gdisc$. Again, they are bilipschitz embeddings: The map $\Psi_i$ is essentially the map $\vp \mapsto \ca\vp$, where $\ca = 1-\alpha/2\pi$. 
Since $\alpha\in (-\pi,\pi)$, $\ca \in (1/2,3/2)$, and thus $\Psi_i$ is injective on any half-annulus, and bilipschitz with constant independent of $\alpha$  and $h$.
From here one proceeds exactly as in \propref{prop:FJMdisloc}.
\end{proof}

We proceed to derive a lower bound for the elastic energy of a thin ring with a disclination defect. This is the analogue of \lemref{lem:lowbound}. 

\begin{lemma}
\label{lem:lowbddisc}
Under the assumptions of \propref{prop:FJMdisc}, there exists a
universal constant $C>0$ such that, for every
$f\in W^{1,2}(\M_{\rho,\rho+h}^h;\R^3)$ and every $Q\in \SO(3)$,
\[
\int_{\M_{\rho,\rho+h}^h} |df-Q\Pdisc|_{\Gdisc,\euc}^2\,\VolGa
\ge
Ch^2 \rho (1-\cos\alpha).
\]
\end{lemma}

\begin{proof}
Set
\[
A:=df-Q\Pdisc .
\]
Since $|\pl_\vp|_{\Gdisc} = \ca r$, we have
\[
|A|_{\Gdisc,\euc}^2
\ge
\frac{1}{\ca^2 r^2}\,|A(\partial_\vp)|^2 ,
\]
with $\ca\in (1/2,3/2)$.
Using that $\VolGa = \ca r \,dr\,d\vp\,dz$,
\[
\int_{\M_{\rho,\rho+h}^h} |A|_{\Gdisc,\euc}^2\,\VolGa
\ge
c\int_0^h\int_\rho^{\rho+h}
\frac{1}{r}\int_0^{2\pi}|A(\partial_\vp)|^2\,d\vp\,dr\,dz .
\]
Arguing as in \lemref{lem:lowbound}, we obtain that for a.e.\ $(r,z)\in(\rho,\rho+h)\times(0,h)$, 
\[
\int_0^{2\pi}|A(\partial_\vp)|^2\,d\vp
\ge
\frac{1}{2\pi}
\left| \int_0^{2\pi}\Pdisc(\partial_\vp)\,d\vp \right|^2 .
\]

It remains to compute the circulation of $\Pdisc$ around the circle
$\{r=\text{const.},z=\text{const.}\}$. By definition,
\[
\Pdisc(\partial_\vp)
=
-\ca r\sin(\ca\vp)\, e_1+\ca r\cos(\ca\vp)\, e_2 .
\]
Integrating over $\vp$ and substituting the definition of $\ca$ gives
\[
\begin{split}
\int_0^{2\pi}\Pdisc(\partial_\vp)\,d\vp
&=
r\brk{\cos(2\pi\ca)-1}\, e_1
+
r\sin(2\pi\ca)\, e_2  \\
&=
r\brk{(\cos\alpha-1)\, e_1-\sin\alpha\,e_2}.
\end{split}
\]
Consequently,
\[
\left|
\int_0^{2\pi}\Pdisc(\partial_\vp)\,d\vp
\right|^2
=
r^2\brk{(\cos\alpha-1)^2+\sin^2\alpha}
=
2r^2(1-\cos\alpha).
\]
Substituting this estimate into the previous inequalities, and absorbing the numerical constants into $C$, we conclude the proof.
\end{proof}

We now complete the proof of the lower bound in \thmref{thm:disclination}:

\begin{proof1}{of lower bound in \thmref{thm:disclination}:}
Let $f\in W^{1,2}(\Mdisc ;\R^3)$ be arbitrary. 
Using the fact that the total volume of $\Mdisc$ with respect to
$\Gdisc$  is comparable to $h$,
\[
\Edisc(f)
\ge
\frac{C}{h}
\int_{\Mdisc}
\dist^2(df\circ (\Pdisc)^{-1},\SO(3))\,\VolGa .
\]

We now restrict the integral to disjoint annuli of width $h$. Let
\[
m_h:=\left\lfloor \frac1h\right\rfloor-1,
\Textand
r_i:=ih,
\qquad
i=1,\ldots,m_h .
\]
Since $h<1/2$, we have $m_h\ge 1$, and so
\[
\Edisc(f)
\ge
\frac{C}{h}
\sum_{i=1}^{m_h}
\int_{\M_{r_i,r_i+h}^h}
\dist^2(df\circ (\Pdisc)^{-1},\SO(3))\,\VolGa .
\]

Fix $i\in\{1,\ldots,m_h\}$. By \propref{prop:FJMdisc}, applied
with $\rho=r_i$, there exists a $Q_i\in\SO(3)$ such that
\[
\int_{\M_{r_i,r_i+h}^h}
\dist^2(df\circ (\Pdisc)^{-1},\SO(3))\,\VolGa
\ge
C\,\frac{h^2}{r_i^2}
\int_{\M_{r_i,r_i+h}^h}
|df-Q_i\Pdisc|_{\Gdisc,\euc}^2\,\VolGa .
\]
Using \lemref{lem:lowbddisc} with $Q=Q_i$, we obtain
\[
\int_{\M_{r_i,r_i+h}^h}
\dist^2(df\circ \Pdisc^{-1},\SO(3))\,\VolGa
\ge C(1-\cos\alpha)\frac{h^4}{r_i}.
\]
Summing over the annuli gives
\[
\Edisc(f)
\ge
\frac{C}{h}
\sum_{i=1}^{m_h}
(1-\cos\alpha)\frac{h^4}{r_i} \\
=
C(1-\cos\alpha)h^3
\sum_{i=1}^{m_h}\frac1{r_i}.
\]
Bounding the sum from below by an integral,
\[
\Edisc(f)
\ge
C(1-\cos\alpha)h^2\log\frac1h .
\]
Finally, for $\alpha\in(-\pi,\pi)$,
\[
1-\cos\alpha
=
2\sin^2\left(\frac{\alpha}{2}\right)
\ge
C\alpha^2 .
\]
Hence
\[
\Edisc(f)
\ge
C\alpha^2 h^2\log\frac1h .
\]
Taking the infimum over $f\in W^{1,2}(\Mdisc;\R^3)$ yields the lower
bound.
\end{proof1}

\section{Upper bounds}
\label{sect:upbounds}

In this section, we prove the upper bounds in \thmref{thm:disclination} and \thmref{thm:dislocation}.
Each bound hinges on explicit constructions of configurations. 
As described in the introduction, these configurations are constructed for the 2D mid-surface.  
We estimate their stretching and bending energy, which together bound the 3D energy as detailed in Appendix~\ref{sec:2D3D}.
More precisely,  denote the mid-surface by $\S$, the 2D prestrain by $\K$, and the surface metric by $g$.
We either construct immersions 
\[
\sigma:\S\to\R^3,
\]
estimate their corresponding plate energy $\EPlate_{(S,\K)}(\sigma)$ as defined in \eqref{eq:EPlate}, and hinge upon \propref{prop:kirchhoff_love} to conclude that their Kirchhoff-Love extensions  $f:\M\to\R^3$ satisfy
\[
E_{(\M,\P)}(f) \le 2\, \EPlate_{(\S,\K)}(\sigma).
\]
Alternatively, we construct maps $\sigma\in W^{2,2}(\S;\R^3) \cap W^{1,\infty}(\S;\R^3)$ 
(which may fail to be immersions), estimate their modified plate energy  
$\tEP_{(\S,\K)}$ as defined in \eqref{eq:mEPlate}, and hinge upon \propref{prop:kirchhoff_love2} to conclude the existence of an extension $f\in W^{1,2}(\M;\R^3)$ satisfying
\[
E_{(\M,\P)}(f) \le \max\brk{68, 4 \|d\sigma\|_\infty^2} \, \tEP_{(\S,\K)}(\sigma) + Ch^2 \|d\sigma\|_4^4.
\]

For consistency of  notations, we denote by $\Sdisc$ and $\Sdisl$ the respective mid-surfaces of $\Mdisc$ and $\Mdisl$. 
The reduced  two-dimensional prestrains are denoted by $\Kdisc$ and $\Kdisl$, respectively, and the induced two-dimensional metrics are denoted by $\gdisc$ and $\gdisl$.

\subsection{Disclinations}
\label{subsect:discupper}

\subsubsection{The flat configuration}

We begin by deriving a simple upper bound, which can be obtained directly in the three-dimensional setting. 

\begin{proposition}
\label{prop:flat_disloc}
Let $h\in (0,1/2)$ and $\alpha\in (-\pi,\pi)$, and let $\Edisc$ be the elastic energy associated with a disclination.
Then the flat configuration $f$ of $\Mdisc$ yields an energy bound
\[
\Edisc(f) \leq C \alpha^2.
\]
for some absolute constant $C>0$. 
\end{proposition}

\begin{proof}
Recall that $\Mdisc=\{(x',x_3) ~:~ |x'| \in (0, 1),\, x_3\in (0,h)\}$, and that the cone metric $\Gdisc$ is given by \eqref{eq:Gdisc}.
Consider the identity map $f = \Mdisc \to \R^3$ in these coordinates, which yields 
\[
df = I_{3\times 3}
\qquad\text{whereas}\qquad
\Gdisc = I_{3\times 3} + (\ca^2-1) r^2 \, d\vp\otimes d\vp.
\]
By the isotropy of the energy density,
\[
\dist(df\circ \Pdisc^{-1},\SO(3)) = \dist(df\circ \Gdisc^{-1/2},\SO(3)) \le |df\circ\Gdisc^{-1/2} - I_{3\times 3}|_{\euc} \le C|1-c_\alpha^2| \le C\alpha. 
\]
The bound is obtained by squaring and integrating.
\end{proof}

\subsubsection{The zero stretching configuration}
\label{sec:iso_disclination}

We now construct a configuration $f:\Mdisc\to \R^3$ satisfying 
\[
\Edisc(f) \leq C|\alpha| h^2 \log(1/h),
\]
thus concluding the proof of the upper bound in \thmref{thm:disclination}.
Following the path outlined at the beginning of this section, we construct configurations $\sigma_h :\Sdisc\to \R^3$  for the mid-surface $\Sdisc = \{x' ~:~ |x'|\in (0,1)\}$ such that their  plate energy is of that order.
We also denote $\Sdiscrho =  \{x' ~:~ |x'|\in (\rho,1)\}$.
The restriction of $\Gdisc$ to $\Sdisc$ is given by
\[
\gdisc = dr \otimes dr + c_\alpha^2 r^2 d\vp \otimes d\vp.
\]

\begin{proposition}
\label{prop:first}
Let $\alpha\in(-\pi,\pi)$, and $\ca = 1 - \alpha/2 \pi$. 
Let $\Psi:\bbS^1\to\R^3$ be a parametrized loop satisfying
\[
\ip{\Psi(t),\Psi(t)} = 1
\Textand
\ip{\Psi'(t),\Psi'(t)} = \ca^2.
\]
That is, $\Psi$ has constant speed $\ca$,  its image lies on the unit sphere and has total length $2\pi \ca$. Then the map (in polar coordinates)
\[
\sigma(r,\vp) = r\,\Psi(\vp)
\]
is an isometric immersion $(\Sdisc,\gdisc)\to (\R^3,\euc)$.
\end{proposition}

\begin{proof}
This is immediate as
\beq
\partial_r \sigma(r,\vp)  = \Psi(\vp)
\Textand
\partial_\vp \sigma(r,\vp)  = r\,\Psi'(\vp),
\label{eq:deriv_sigma}
\eeq
hence
\[
\begin{aligned}
\ip{\partial_r\sigma,\partial_r\sigma} &= \ip{\Psi(\vp),\Psi(\vp)} = 1 \\
\ip{\partial_r\sigma,\partial_\vp \sigma} &=  r\, \ip{\Psi(\vp),\Psi'(\vp)} = 0 \\
\ip{\partial_\vp \sigma,\partial_\vp \sigma} &= r^2 \ip{\Psi'(\vp),\Psi'(\vp)} = \ca^2 r^2,
\end{aligned}
\]
where we used the fact that $\Psi(\vp)$ and $\Psi'(\vp)$ are orthogonal, since $\Psi(\vp)$ is a unit vector.
\end{proof}

\begin{proposition}
Let $\sigma$ be as in \propref{prop:first} and denote by $\frakn$ the unit normal of $\sigma$. 
Then,
\[
|d\frakn|_{\gdisc}^2 = \frac{1}{\ca^4 r^2} |\Psi(\vp)\times \Psi''(\vp)|^2.
\]
\end{proposition}

\begin{proof}
Using once again the orthogonality of $\Psi(\vp)$ and $\Psi'(\vp)$,
\[
|\Psi(\vp)\times\Psi'(\vp)| =   |\Psi(\vp)|\,|\Psi'(\vp)| = \ca.
\]
Therefore, using \eqref{eq:deriv_sigma},  the unit normal is given by
\[
\frakn(r,\vp) = \frac{1}{\ca} \Psi(\vp)\times \Psi'(\vp).
\]
Hence, the claim follows from
\[
d\frakn = \frac{d\vp}{\ca} \Psi(\vp)\times \Psi''(\vp) = \frac{1}{\ca^2 r} \Psi(\vp)\times \Psi''(\vp) \,(\ca r\,d\vp),
\]
noting that $\ca r\,d\vp$ is of unit $\gdisc$-norm.
\end{proof}

\begin{corollary}
\label{cor:Ca}
Let $\sigma$ be as in \propref{prop:first} and let $\rho\in (0,1)$. 
Then the bending energy of $\sigma$ as defined by \eqref{eq:EB}, restricted to $\Sdiscrho$, is
\[
\EB_{(\Sdiscrho,\gdisc)}(\sigma) = \int_{\Sdiscrho} |d\frakn|^2_{\gdisc} \,\dVol{\gdisc} = C(\alpha) \, \log\frac{1}{\rho},
\]
where
\begin{equation}
\label{eq:defC}
C(\alpha) = \frac{1}{\ca^3} \int_0^{2\pi} |\Psi(\vp)\times \Psi''(\vp)|^2\,d\vp.
\end{equation}
\end{corollary}

With this we can immediately calculate the bending energy of standard truncated cones:

\begin{corollary}[Isometric immersion for $\alpha>0$]
\label{cor:positive_cone_ansatz}
Let $\alpha\in (0,\pi)$, and let 
\beq
\Psi(t) = \brk{\ca\,\cos t, \ca \,\sin t, (1 - \ca^2)^{1/2}}.
\label{eq:disclination_example1}
\eeq
be the standard isometric immersion of a cone.
Then the bending energy satisfies 
\[
\EB_{(\Sdiscrho,\gdisc)}(\sigma) =  \frac{2\pi}{\ca}\brk{\frac{\alpha}{\pi} - \frac{\alpha^2}{4\pi^2}}.\, \log\frac{1}{\rho} \le C\,\alpha \,\log\frac{1}{\rho}
\]
for some universal constant $C>0$.
\end{corollary}

%

We proceed to construct an isometric immersion for the case of negative disclinations, i.e., $\alpha<0$.
In order to do so, we first observe that we can relax the assumption that $\Psi$ is of constant speed.
Let $t\mapsto s(t)$ be a reparametrization of $\bbS^1$ satisfying $s'(t) <3$. 
For $\Psi:\bbS^1\to\R^3$, let $\Phi:\bbS^1\to\R^3$ be defined by
\[
\Psi(t) = \Phi(s(t)).
\]
Then, a direct calculation shows that
\[
\Psi(t) \times \Psi''(t) = \Phi(s(t)) \times \Phi''(s(t))\, s'^2(t),
\]
from which follows that
\beq
\label{eq:Ca_Phi}
\begin{split}
C(\alpha) 
&=  \frac{1}{\ca^3} \int_0^{2\pi} | \Phi(s(t)) \times \Phi''(s(t))|^2 s'^2(t)\,dt \le \frac{3}{\ca^3} \int_0^{2\pi} | \Phi(t) \times \Phi''(t)|^2 \,dt,
\end{split}
\eeq
implying that $C(\alpha)$ can be evaluated up to a multiplicative constant by using the non-constant-speed curve $\Phi$,
which only needs to lie on the unit sphere and have total length $2\pi\ca$.


\begin{proposition}[Isometric immersion for $\alpha<0$]
\label{prop:negative_cone_ansatz}
Let $\alpha\in (-\pi,0)$, and let 
\[
\Phi(s) = (\cos\theta(s) \cos s ,\cos\theta(s) \sin s, \sin\theta(s)),
\]
where 
\[
\theta(s) = A_\alpha\,\sin 2s,
\]
such that $A_\alpha$ satisfies
\beq
\label{eq:A_alpha}
2\pi c_\alpha = \int_0^{2\pi} \brk{\cos^2(A_\alpha\,\sin 2s) + 4 A_\alpha^2 \cos^2 2s}^{1/2}\,ds.
\eeq
Then, the constant speed reparametrization $\Psi$ of $\Phi$ defines an isometric immersion $\sigma$ of $(\Sdisc,\gdisc)$, whose bending energy, restricted to $\Sdiscrho$, satisfies
\[
\EB_{(\Sdiscrho,\gdisc)}(\sigma)  \le C\,|\alpha| \,\log\frac{1}{\rho}.
\]
\end{proposition}

\begin{figure}[h]
\begin{center}
\includegraphics[height=2.0in]{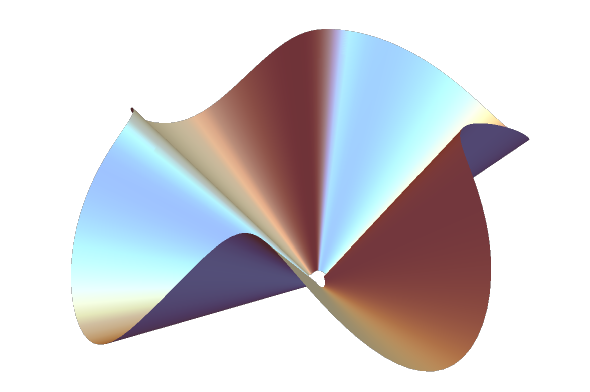}
\end{center}
\caption{A depiction of the isometric immersion for $e$-cones defined in Proposition~\ref{prop:negative_cone_ansatz}.}
\label{fig:Econe}
\end{figure}

\begin{proof}
First, we note that 
\[
|\Phi'(s)|^2 = \cos^2 \theta(s) + \theta'^2(s)
= \cos^2\brk{A_\alpha\,\sin 2s} + 4 A_\alpha^2\,\cos^2 2s,
\]
i.e., it follows from \eqref{eq:A_alpha} that $\Phi$ has total length $2\pi\ca$, hence $\sigma$ is an isometric immersion.

Next, we bound the speed of $\Phi$, so that \eqref{eq:Ca_Phi} can be used.
Denote the right-hand side of \eqref{eq:A_alpha} by $I(A_\alpha)$.
It is immediate that 
\[
2\pi c_\alpha = I(A_\alpha) \ge 2A_\alpha \int_0^{2\pi} |\cos(2s)|\,ds = 8A_\alpha,
\]
hence $A_\alpha \le (\pi/4) c_\alpha\le 3\pi/8$, where we have used the fact that $\ca \in [1,3/2]$.
Denote $a=\sin 2s$, so that 
\[
|\Phi'(s)|^2 = \cos^2(A_\alpha a) + 4A_\alpha^2 (1-a^2).
\]
It is easy to see this is a decreasing function of $a$ for $A_\alpha$ in this range.
Thus,
\[
|\Phi'(s)|^2 \ge \cos^2A_{\alpha} \ge  \cos^2\brk{\frac{3\pi}{8}} \approx 0.382,
\]
hence $|\Phi'(s)|^{-1} < 3$.

We now estimate the bending energy. 
Differentiating $\Phi(s)$ twice we find
\[
\begin{aligned}
|\Phi(s)\times\Phi''(s)|^2 &= (\theta''(s) + \sin\theta(s)\cos\theta(s))^2 + 4(\theta'(s))^2 \sin^2 \theta(s) 
= O(A_\alpha^2).
\end{aligned}
\]
Taylor expanding \eqref{eq:A_alpha} for small values of $A_\alpha$, we obtain that for small values of $|\alpha|$,
\[
c_\alpha = 1 + \frac34 A_\alpha^2 + O(A_\alpha^4)
\]
and thus, for small $|\alpha|$, 
\[
|\Phi(s)\times\Phi''(s)|^2 \sim A_\alpha^2 \sim |\alpha|.
\]
The bound on $\EB_{(\Sdiscrho,\gdisc)}(\sigma)$ now follows from \corrref{cor:Ca} and \eqref{eq:Ca_Phi}.
\end{proof}


%

Combining \corrref{cor:positive_cone_ansatz} and \propref{prop:negative_cone_ansatz} we have thus obtained:

\begin{corollary}
\label{cor:EB_sigma_disc}
Let $\alpha\in(-\pi,\pi)$ and $\rho \in (0,1)$. 
Then there exists an isometric immersion $\sigma$ of $(\Sdisc,\gdisc)$ in $\R^3$, such that
\[
\EB_{(\Sdiscrho,\gdisc)}(\sigma) \leq C |\alpha| \,\log\frac{1}{\rho},
\]
for some constant $C>0$ independent of $\alpha$ and $\rho$. 
\end{corollary}


We now regularize the isometric immersion $\sigma$ near the apex of the cone to complete the proof of \thmref{thm:disclination}:

\begin{proposition}
\label{prop:disc_ub}
Let $\alpha \in (-\pi,\pi)$ and $h<1/2$.
Then, for some universal constant $C>0$, there exists a configuration $\sigma_h:\Sdisc\to \R^3$ satisfying $|d\sigma_h| < C$, and such that
\beq\label{eq:s_plus_b_disloc}
\tEP_{(\S,\Kdisc)} (\sigma_h) \le C|\alpha|h^2 \log\frac{1}{h}.
\eeq
It follows from \propref{prop:kirchhoff_love2} that there exist configurations $f_h\in W^{1,2}(\Mdisc;\R^3)$ such that $\Edisc(f_h) \le C|\alpha|h^2 \log(1/h)$.
\end{proposition}

\begin{proof}
Fix a smooth nondecreasing function $\eta:[0,\infty)\to[0,1]$ such that
$\eta=0$ on $[0,1/2]$ and $\eta=1$ on $[1,\infty)$. We treat positive
and negative disclinations separately.

\emph{Case 1: $\alpha>0$.}
Set
\[
q_\alpha:=\sqrt{1-c_\alpha^2}.
\]
For $q\in[0,q_\alpha]$, define
\[
\Psi_q(\vp):=
\big(\sqrt{1-q^2}\cos\vp,\sqrt{1-q^2}\sin\vp,q\big).
\]
Then $\Psi_0(\vp)=(\cos\vp,\sin\vp,0)$, while
$\Psi_{q_\alpha}$ is the curve in
\eqref{eq:disclination_example1}. Moreover,
\[
q_\alpha^2=1-c_\alpha^2\le C\alpha.
\]
Define
\[
q_h(r):=\eta(r/h)q_\alpha,
\qquad
\sigma_h(r,\vp):=r\Psi_{q_h(r)}(\vp).
\]
Thus $\sigma_h$ is the flat configuration for $r<h/2$ and agrees with
the conical isometric immersion of
\corrref{cor:positive_cone_ansatz} for $r>h$.
Since $|q_h'|\le Cq_\alpha/h$, direct differentiation gives, for $r<h$,
\[
\dist^2\big(d\sigma_h\circ\Kdisc^{-1},O(2,3)\big)
   \le Cq_\alpha^2,
\qquad
|d\frakn_h|_{\gdisc}^2
   \le C\frac{q_\alpha^2}{h^2},
\]
where $\frakn_h$ is the unit normal to $\sigma_h$. Since
$\Vol_{\gdisc}(\{r<h\})\le Ch^2$, it follows that
\[
\ES_{(\Sdisc,\Kdisc)}(\sigma_h)
+h^2\EB_{\{r<h\}}(\sigma_h)
\le C\alpha h^2.
\]
On $\Sdisch$, $\sigma_h$ is an isometric immersion, and
\corrref{cor:EB_sigma_disc} gives
\[
h^2\EB_{(\Sdisch,\gdisc)}(\sigma_h)
\le C\alpha h^2\log\frac1h.
\]

\emph{Case 2: $\alpha<0$.}
Let $A_\alpha>0$ be the parameter introduced in
\propref{prop:negative_cone_ansatz}, namely, a solution of
\eqref{eq:A_alpha},
\[
2\pi c_\alpha
=
\int_0^{2\pi}
\left(
\cos^2(A_\alpha\sin 2s)
+4A_\alpha^2\cos^2 2s
\right)^{1/2}\,ds.
\]
For $A\in[0,A_\alpha]$, let
\[
\Phi_A(s)
=
\big(
\cos(A\sin 2s)\cos s,\,
\cos(A\sin 2s)\sin s,\,
\sin(A\sin 2s)
\big),
\]
and denote by $\Psi_A$ its constant-speed parametrization. Thus
$\Psi_0(\vp)=(\cos\vp,\sin\vp,0)$, while $\Psi_{A_\alpha}$ is the curve
constructed in \propref{prop:negative_cone_ansatz}. The estimates in
the proof of that proposition give
\[
A_\alpha^2\le C|\alpha|.
\]
Moreover, since $A_\alpha\le 3\pi/8$ and $|\Phi_A'|$ is uniformly
bounded away from zero, the family $\Psi_A$ depends smoothly on $A$,
with uniform $C^2$ bounds; in particular,
\[
\|\Psi_A-\Psi_0\|_{C^2}\le CA.
\]
Set
\[
A_h(r):=\eta(r/h)A_\alpha,
\qquad
\sigma_h(r,\vp):=r\Psi_{A_h(r)}(\vp).
\]
As in the positive case, direct differentiation yields, for $r<h$,
\[
\dist^2\big(d\sigma_h\circ\Kdisc^{-1},O(2,3)\big)
   \le CA_\alpha^2,
\qquad
|d\frakn_h|_{\gdisc}^2
   \le C\frac{A_\alpha^2}{h^2}.
\]
Hence the contribution of $\{r<h\}$ to $\EPlate$ is bounded by
$C|\alpha|h^2$. On $\Sdisch$, $\sigma_h$ agrees with the isometric
immersion of \propref{prop:negative_cone_ansatz}, and therefore
\[
h^2\EB_{(\Sdisch,\gdisc)}(\sigma_h)
\le C|\alpha|h^2\log\frac1h.
\]

Thus, both in the case $\alpha>0$ and $\alpha<0$,
\[
\EPlate_{(\Sdisc,\Kdisc)}(\sigma_h)
\le C|\alpha|h^2\left(1+\log\frac1h\right)
\le C|\alpha|h^2\log\frac1h,
\]
where we used $h<1/2$. Finally, \propref{prop:kirchhoff_love} gives a
Kirchhoff--Love extension $f_h$ satisfying
\[
\Edisc(f_h)
\le 2\EPlate_{(\Sdisc,\Kdisc)}(\sigma_h)
\le C|\alpha|h^2\log\frac1h.
\]
\end{proof}

\subsection{Dislocations}
\label{subsect:dislupper}

We  proceed to prove the upper bounds for thin elastic sheets with edge-dislocations. 

\subsubsection{The flat configuration}

Like for disclinations, we start by proving an upper bound for a flat configuration, using  the three-dimensional setting.

\begin{theorem}
\label{thm:disloc_ub_e_ll_h}
Let $\Edisl$ be  the energy associated with a dislocation. Then there exists a constant $C$ independent of $\e$ and $h$, such that
\[
\inf \Edisl \leq C \e^2\,\log\frac{1}{\e}.
\]
\end{theorem}

\begin{proof}
Take the coordinate inclusion map $f = \iota : \Mdisl \to \R^3$, as described in \propref{prop:prelim}.
By \eqref{eq:disc_iota_est}, 
\[
|\Pdisl - d\iota|_{\Gdisl,\euc} \lesssim  \frac{\e}{r},
\]
hence
\[
\dist^2(d\iota\circ\Pdisl,\SO(3)) \lesssim \frac{\e^2}{r^2}.
\]
Using the bilipschitz equivalence between $(\Mdisl,\Gdisl)$ and $(\Mdisl,\iota^*\euc)$,  
\[
\int_{\Mdisl} \dist^2(d\iota\circ\Pdisl,\SO(3))\,\VolGe \le C \int_{\Mdisl} \dist^2(d\iota\circ\Pdisl,\SO(3))\circ\iota^{-1}\,\VolE \le C h\e^2\,\log\frac{1}{\e}.
\]
Dividing by the volume of the body, we obtain the required bound.
\end{proof}


\subsubsection{The zero stretching configuration}

In this subsection, we prove an upper bound based on an isometric immersion of the mid-surface, similarly to the bound obtained for disclinations. Unlike the in the disclination case, there is no need to regularize the solution by allowing for stretching in the vicinity of an apex.

One such construction is presented in \cite{GHKM13}, connecting  a positive disclination and a negative disclination of opposite charge whose apexes do not coincide via two flat slabs. 
An even more explicit construction, albeit a ``non-physical'' one, yielding an immersion rather than an embedding, yields a configuration resembling a double-folded cone. We present here the latter construction.
 
\begin{theorem}
\label{thm:disloc_ub_h_ll_e}
Let $\Edisl$ be the energy associated with a dislocation. There exists a constant $C$ independent of $\e$ and $h$, such that
\[
\inf \Edisl \leq C h^2\,\log\frac{1}{\e}.
\]
\end{theorem}

\begin{proof}
Following the procedure outlined in the beginning of this section, it suffices to construct an isometric immersion $\sigma:\Sdisl\to\R^3$ whose bending energy is bounded from above by
\[
\EB_{(\Sdisl,\Kdisl)}(\sigma) \le C\, \log\frac{1}{\e}
\] 
for some constant $C$.

The two-dimensional prestrain map is given  in polar coordinates by 
\[
\Kdisl =  e_1\otimes (\cos\vp\, dr - r\sin\vp\,d\vp)  + e_2\otimes (\sin\vp\, dr + r\cos\vp\,d\vp)  +  \frac{\e}{2\pi} e_1\otimes d\vp(x,y,z),
\]
yielding the following two-dimensional metric 
\[
(\partial_r,\partial_r)_{\gdisl} = 1
\qquad
(\partial_r,\partial_\vp)_{\gdisl} = \frac{\e}{2\pi}\cos\vp
\qquad
(\partial_\vp,\partial_\vp)_{\gdisl} = r^2 - \frac{\e}{\pi} r\,\sin\vp + \brk{\frac{\e}{2\pi}}^2.
\]
Our ansatz consists of immersions for which the radial geodesics are asymptotic curves, i.e., map into spatial geodesics,
\[
\sigma(r,\vp) = \e\, \gamma(\vp) + \nu(\vp)(r - \e),
\]
where $\gamma:\bbS^1\to\R^3$ and $\nu:\bbS^1\to\bbS^2$. 
A direct calculation yields that the pullback metric $\sigma^*\euc$ is given by
\[
(\partial_r,\partial_r)_{\sigma^*\euc}  = 1
\qquad
(\partial_r,\partial_\vp)_{\sigma^*\euc}  = \e (\nu,\gamma')_\euc
\qquad
(\partial_\vp,\partial_\vp)_{\sigma^*\euc}  = \e^2 |\gamma'|_\euc^2 + 2 \e (\gamma',\nu')_\euc(r - \e) +  |\nu'|_\euc^2(r - \e)^2,
\]
using the fact that $(\nu,\nu')_\euc = 0$.
Equating $\gdisl =  \sigma^*\euc$ and matching same powers of $r$, we obtain the following set of equations.
\[
\begin{aligned}
& (\nu,\gamma')_\euc = \frac{1}{2\pi}\cos\vp \\
& |\nu'|_\euc^2 = 1 \\
& (\gamma',\nu')_\euc  = 1 - \frac{1}{2\pi} \,\sin\vp \\
& |\gamma'|_\euc^2  = 2 (\gamma',\nu')_\euc  - 1 +   \brk{\frac{1}{2\pi}}^2.
\end{aligned}
\]

Note that
\[
(\gamma',\nu)_\e^2 +  (\gamma',\nu')_\e^2  = \frac{1}{4\pi^2}\cos^2\vp + \brk{1 - \frac{1}{2\pi} \,\sin\vp}^2 = |\gamma'|_\e^2.
\]
Since both $\nu$ and $\nu'$ are unit vectors, this implies that
\[
\gamma' =  \frac{1}{2\pi}\cos\vp  \, \nu + \brk{ 1 - \frac{1}{2\pi} \,\sin\vp}\, \nu'.
\]
Thus, we need to find a mapping $\nu:\bbS^1\to\bbS^2$, whose derivative has unit magnitude, such that the integral of $\gamma'$ vanishes. Integrating by parts, we obtain the following constraint
\beq
\int_{\bbS^1} \cos\vp  \, \nu \, d\vp = 0.
\label{eq:closed_gamma}
\eeq
While the constraint that radial lines are asymptotic curves is quite restrictive, there are multiple solutions, none of which, however, is an embedding (condition \eqref{eq:closed_gamma} annihilates the first Fourier mode). For example,
\[
\nu(\vp) = \half \mymat{\cos2\vp & \sin2\vp & \sqrt{3}},
\]
is a valid solution. The corresponding immersion is shown in \figref{fig:2fold}.

\begin{figure}
\centerline{\includegraphics[height=2.2in]{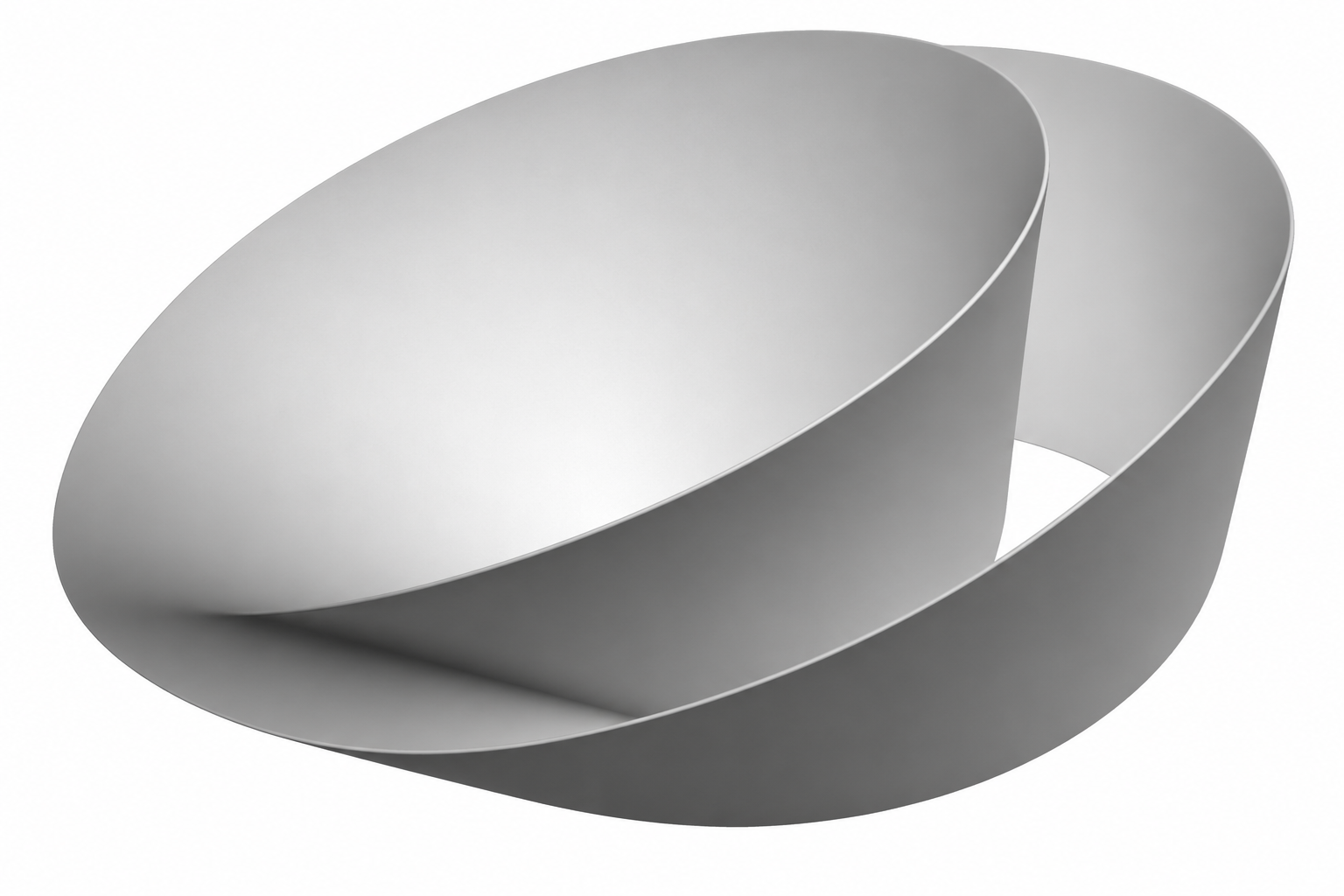}}
\caption{A surface with a dislocation immersed isometrically into $\R^3$.
Note that even for $\e=0$, a punctured plane may be immersed in $\R^3$ as a double-folded cone. }
\label{fig:2fold}
\end{figure}

It remains to show that there exists a constant $C>0$ independent of $\e$, such that
\[
\EB_{(\Sdisl,\Kdisl)}(\sigma) \leq C\,\log\frac{1}{\e}.
\]
An explicit calculation yields the same logarithmic divergence in $1/\e$, as for disclinations. 

\end{proof}
 
\subsubsection{The single fold construction}

In this section, we present  an upper bound relevant in a regime in which $h$ and $\e$ are comparable. In the previous cases, the ansatz was either a flat configuration or an isometric immersion of $(\Sdisl,\gdisl)$, and the energy resulted either from pure stretching or from pure bending. 
In the present case, the ansatz has comparable stretching and bending energies; the construction, of which we will only provide a sketch, was derived by Conti \cite{Con23}. 

\begin{theorem}[\cite{Con23}]
\label{thm:disloc_ub_h_sim_e}
Let $\Edisl$ be the energy associated with a dislocation.
There exists a constant $C$ independent of $\e$ and $h$, such that
\[
\inf \Edisl \leq C (\e^2 + h^2).
\]
\end{theorem}

For completeness, we give a sketch of the idea of the proof.
Hinging once again on \thmref{thm:3D2D}, it suffices to construct immersions $\sigma:\Sdisl\to\R^3$ of the mid-surface, for which the estimates
\[
\ES_{(\Sdisl,\Kdisl)}(\sigma) \le C \e^2
\Textand
\EB_{(\Sdisl,\Kdisl)}(\sigma) \le C.
\]
hold for some $C$ independent of $\e$.

The construction is easiest to present for a body with rectangular inner and outer boundaries, as depicted below. 
The application to the body manifold $(\Sdisl,\gdisl)$ is immediate by embedding it in such a rectangle.

\[
\btkz
	\draw (-2,-2) -- (2,-2) -- (2,2) -- (-2,2) -- cycle;
	\draw (-0.5,-0.5) -- (0.5,-0.5) -- (0.5,0.5) -- (-0.5,0.5) -- cycle;
	\tkzText(0,0.8){\footnotesize $2\e$};
	\tkzText(0,-0.8){\footnotesize $2\e$};
	\tkzText(-0.75,0){\footnotesize $\e$};
	\tkzText(0.8,0){\footnotesize $2\e$};
	\tkzText(0,2.3){\footnotesize $2$};
	\tkzText(2.3,0){\footnotesize $2$};
	\tkzText(0,-2.3){\footnotesize $2$};
	\tkzText(-2.5,0){\footnotesize  $2-\e$};
	
	\begin{scope}[xshift=7cm]
	\draw[fill=gray!10] (-2,-1.9) -- (2,-1.9) -- (0.5,-0.4) -- (-0.5,-0.4) -- (-0.5,0.4)  -- (0.5,0.4) -- (2,1.9) -- (-2,1.9) -- cycle;
	\draw[dashed](0.5,-0.4) -- (0.5,0.4);
	\tkzText(0,0.7){\footnotesize $2\e$};
	\tkzText(0,-0.7){\footnotesize $2\e$};
	\tkzText(-0.75,0){\footnotesize $\e$};
	\tkzText(0,2.2){\footnotesize $2$};
	\tkzText(0,-2.2){\footnotesize $2$};
	\tkzText(-2.5,0){\footnotesize $2-\e$};
	\tkzDefPoint(0.5,0.4){A};
	\tkzDefPoint(0.5,-0.4){B};
	\tkzDrawPoints(A,B);
	\begin{scope}[xshift=5]
	\draw[fill=gray!10] (2,2) -- (2,-2) -- (0.5,-0.5) -- (0.5,0.5) -- cycle;
	\tkzText(0.8,0){\footnotesize $2\e$};	
	\tkzText(2.3,0){\footnotesize $2$};	
	\end{scope}
	\end{scope}
\etkz
\]

The figure on the right depicts the dislocated surface: a geodesic rectangle, whose right side is longer than its left side by the dislocation's magnitude $\e$. The figure on the left depicts a dissection of the surface into two simply-connected domains. 
An immersion into $\R^3$ is sought such that three-quarters of the surface remain planar, whereas the remaining quarter,
which has an excess of length, forms a fold as depicted below:

\begin{center}
\includegraphics[height=1.5in]{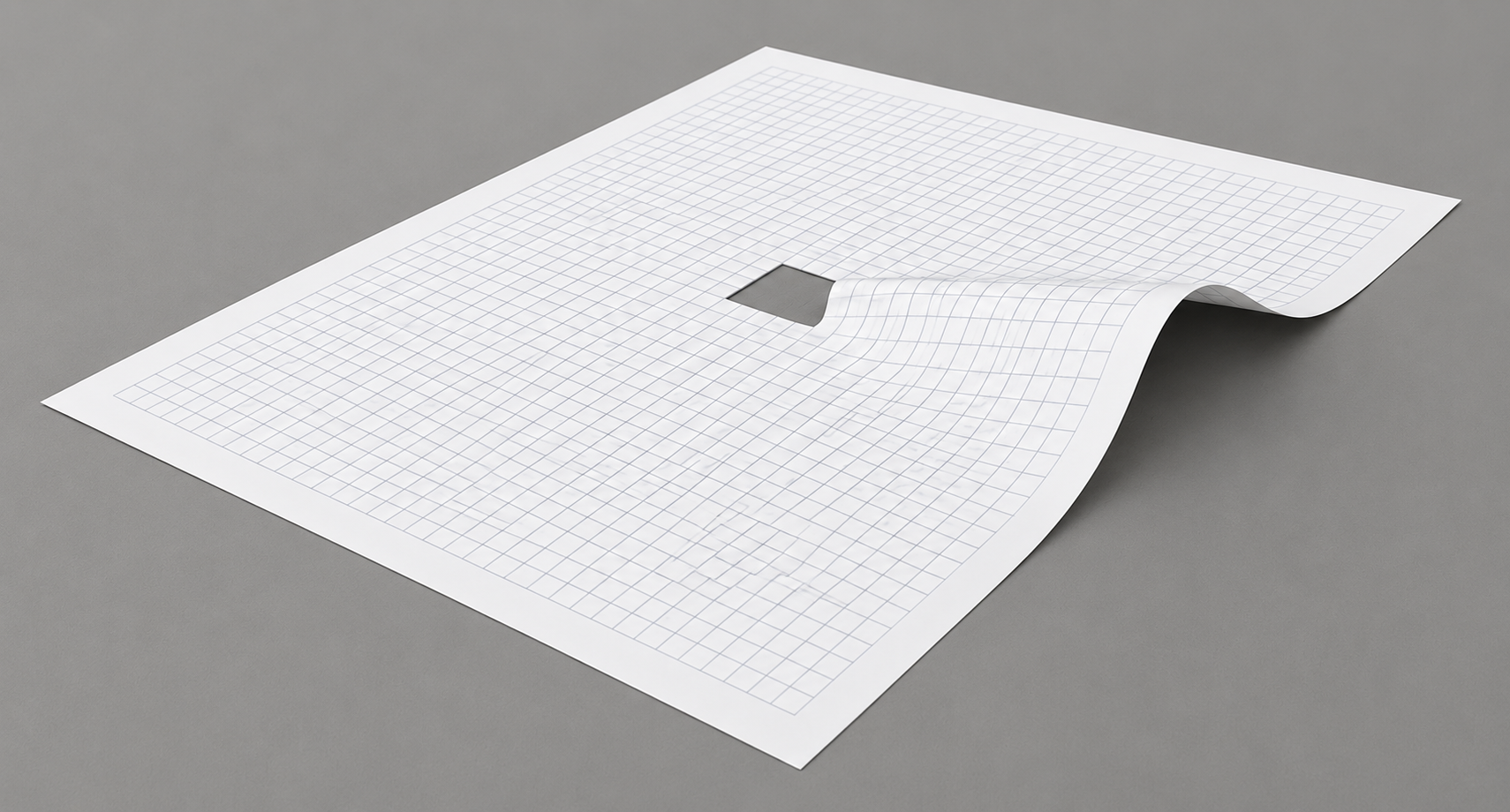}
\end{center}

The construction of \cite{Con23} shows that a fold having fixed profile of variable width and height can yield the desired scalings of stretching and bending energies.


\appendix
\section{Relating three-dimensional and two-dimensional models}
\label{sec:2D3D}

In this section we develop inequalities relating the 3D elastic models and their corresponding ``stretching plus bending'' 2D models. 

\subsection{``Stretching plus bending'' models}
Consider a thin elastic body $\M^h = \S\times (0,h)$ where $\S \subset \R^2$ is the mid-surface.\footnote{It is in fact the bottom-surface, but we refer to it as the mid-surface without redefining the domain with $z\in(-h/2,h/2)$. The only implication of considering a ``true'' mid-surface is changing below some numerical constants.}
We denote points in $\M^h$ by $x=(x',x_3)$, where $x'\in \S$.
Let $\P$ be a prestrain map $\P: T\M\to \R^3$ of the form
\beq\label{eq:PK}
\P_{(x',x_3)} = \mymat{\K_{x'} & 0 \\ 0 & 1},
\eeq
where $\K:T\S\to \R^2$.
This represents a non-Euclidean plate, as the second fundamental form of $\S$ in $\M^h$ (with respect to the  metric $\g$ induced by $\P$, see \eqref{eq:metric}), vanishes.
Denote by $\g$ the metric induced on $\S$ by $G$.

We now define two ``stretching plus bending'' energies related to $(\S,\K)$:

First, for an immersion $\sigma:\S\to \R^3$, whose oriented unit normal is $\frakn: \S\to \R^3$, we define
\beq
\EPlate_{(\S,\K)} (\sigma) =  \ES_{(\S,\K)} (\sigma) + h^2\, \EB_{\S}(\sigma),
\label{eq:EPlate}
\eeq
where 
\beq
\label{eq:ES}
\ES_{(\S,\K)} (\sigma) = \dashint_{\S} \dist^2(d\sigma\circ \K^{-1},O(2,3))\,\dVol{\g},
\eeq
and
\beq
\EB_{\S}(\sigma) = \frac13 \dashint_{\S} |d\frakn|_\g^2\,\dVol{\g}.
\label{eq:EB}
\eeq
The first term is called the \Emph{stretching energy}, whereas the second term is called the \Emph{bending energy}.
When the stretching vanishes, i.e., when $\sigma$ is an isometric immersion, then the $\g$ norm of the Weingarten map $d\frakn$ equals that of the second fundamental form (which is often used as an alternative definition of the bending energy);
in the presence of stretching, however, the Weingarten map is more convenient for analysis (see a detailed discussion in \cite{AKM22}).
This energy is well-defined whenever $\sigma$ is a 2-Sobolev immersion \cite{AKM22}, namely, when $\sigma\in W^{1,2}(\S;\R^3)$,  $d\sigma$ has rank two almost everywhere, so $\frakn$ is well defined, and $\frakn\in W^{1,2}(\S;\R^3)$.

Second, we define a modified energy, that is well-defined for any $\sigma\in W^{2,2}(\S;\R^3)$, even if it is not an immersion: 
\beq
\tEP_{(\S,\K)} (\sigma) =  \ \ES_{(\S,\K)} (\sigma) + h^2\, \tEB_{\S}(\sigma),
\label{eq:mEPlate}
\eeq
where 
\[
\tEB_{\S}(\sigma) = \frac13 \dashint_{\S} |\nabla^{\g} d\sigma|_{\g}^2\,\dVol{\g},
\]
and where $\nabla^{\g}$ is the covariant derivative with respect to the metric $\g$.
This is a geometric version of an energy in which the bending energy density is given by $|\nabla^2 \sigma|_\euc$, where $\nabla^2$ is the Hessian in coordinates (see, e.g., \cite{Olb17}).
The use of the geometric Hessian is important in order to obtain the precise relations below between this energy and the three-dimensional one.
Note that whenever $\sigma$ is an isometric immersion, i.e., when the stretching vanishes,
\[
\EPlate_{(\S,\K)} (\sigma) = \tEP_{(\S,\K)} (\sigma)
\]
(this statement is true only when considering the $\g$-Hessian).

The aim of this section is to prove the following relations:

\begin{theorem}
\label{thm:3D2D}
The following inequalities hold:
\beq
\inf_{W^{1,2}(\M^h;\R^3)} E_{(\M^h,\calP)} \le  
2\, \inf_{\textup{Imm}^2(\S;\R^3)} \EPlate_{(\S,\K)},
\label{eq:main2D3D}
\eeq
where $\textup{Imm}^2(\S;\R^3)$ is the space of 2-Sobolev immersions defined above, and
\[
\inf_{W^{1,2}(\M^h;\R^3)} E_{(\M^h,\calP)} \le
\inf_{\sigma\in W^{2,2}(\S;\R^3) \cap W^{1,\infty}(\S;\R^3)} 
\brk{\max\brk{68, 4 \|d\sigma\|_\infty^2} \, \tEP_{(\S,\K)}(\sigma) + Ch^2 \|d\sigma\|_4^4}
\]
for some $C\ge 0$ depending on $\g$, that vanishes if $\g$ admits non-trivial parallel vector fields (e.g., for bodies with dislocations).
Furthermore, if the Gaussian curvature of $\g$ vanishes, then we have the reverse inequality
\[
\inf_{W^{1,2}(\M^h;\R^3)} E_{(\M^h,\calP)} \ge c
\inf_{W^{2,2}(\S_{2h};\R^3)} \, \tEP_{(\S_{2h},\K)}.
\]
for some $c>0$ depending only on the domain $\S$, and $\S_{2h} = \{x'\in \S ~:~ \dist_{\g}(x', \partial \S)>2h\}$
is the mid-surface minus a $2h$-tubular neighborhood of the boundary.
\end{theorem}

\subsection{Some linear-algebraic relations}

We begin by noting some elementary results in linear algebra that will be needed below.

\begin{lemma}
\label{lem:dicO23}
Let $A\in\R^2\otimes \R^3$. Then, 
\[
\dist^2(A,\O(2,3)) = |(A^TA)^{1/2}-I|^2 =  |A|^2 - 2 \brk{|A|^2 + 2|A(e_1)\times A(e_2)|}^{1/2} + 2
\]
\end{lemma}

\begin{proof}
The first identity follows from polar decomposition. 
Next, we have 
\[
\begin{aligned}
|(A^TA)^{1/2}-I|^2  &= \tr ((A^TA)^{1/2}-I)^2 \\
&= \tr A^TA - 2 \tr(A^TA)^{1/2} + 2 \\
&= |A|^2 - 2 \tr(A^TA)^{1/2} + 2.
\end{aligned}
\]
Since for a symmetric 2-by-2 matrix $G$,
\[
\tr G = \tr^2 G^{1/2} - 2(\det G)^{1/2},  
\]
it follows that
\[
\tr(A^TA)^{1/2} = \brk{|A|^2 + 2(\det(A^TA))^{1/2}}^{1/2}.
\]
Finally, 
\[
\det(A^TA) = |A(e_1)\times A(e_2)|^2.
\]
\end{proof}

\begin{lemma}
\label{lem:cross_half}
Let $A\in\R^2\otimes \R^3$. Then 
\[
\dist^2(A,\O(2,3)) \ge (|A| - \sqrt{2})^2.
\]
In particular, 
\[
|A| \ge 2\sqrt{2}
\quad \text{implies} \quad
|A| \le 2\, \dist(A,\O(2,3)).
\]
Furthermore,
\[
\dist(A,\O(2,3)) \ge 
\begin{cases}
2 (1 - |A(e_1)\times A(e_2)|^{1/2})^2 & |A(e_1)\times A(e_2)|\ge 1/4 \\
1-2|A(e_1)\times A(e_2)| 			& |A(e_1)\times A(e_2)|< 1/4.
\end{cases}
\]
In particular,
\[
|A(e_1)\times A(e_2)| < \frac12
\quad\text{implies} \quad
\dist(A,\O(2,3)) \ge \sqrt{2}-1.
\]
\end{lemma}

\begin{proof}
The first inequality follows from the reverse triangle inequality. As for the second, we note that
\[
|\v\times \w| \le |\v| |\w| \le \frac12(|\v|^2 + |\w|^2), 
\]
i.e.,
\[
|A|^2 = |A(e_1)|^2 + |A(e_2)|^2 \ge 2|A(e_1)\times A(e_2)| \equiv c.
\]
By the second identity of previous lemma, it follows that
\[
\dist^2(A,\O(2,3)) \ge \min_{x\ge c} (x - 2 \brk{x + c}^{1/2}) + 2,
\]
which attains its minimum at $x = c$ if $c\ge 1/2$ and at $x=1-c$ if $c<1/2$, which yields the desired inequality.
\end{proof}

\subsection{Bounding the 3D model from above}
In this section we establish the first two inequalities in \thmref{thm:3D2D}.

\begin{proposition}[Kirchhoff-Love estimate for immersions]
\label{prop:kirchhoff_love}
Let $\sigma:\S\to\R^3$ be a 2-Sobolev immersion.
Then, the extension
\beq
f(x',x_3) = \sigma(x') + x_3\,\frakn(x')
\label{eq:KL_extension}
\eeq
of $\sigma$ satisfies
\beq
E_{(\M^h,\calP)}(f) \le 2\, \EPlate_{(\S,\K)}(\sigma).
\label{eq:kirchhoff_love}
\eeq
\end{proposition}

\begin{proof}
Fix a point $x\in \M^h$.
Let $Q\in \O(2,3)$ be the projection of $A = d\sigma\circ \K^{-1}$ onto $\O(2,3)$.
Since $d\sigma$ is of full rank, $Q = A(A^TA)^{-1/2}$,  and thus $Q$ and $d\sigma$ have the same image. 
In particular, the image of $Q$ is perpendicular to $\frakn$, and thus $U = (Q | \frakn ) \in \SO(3)$.

Using the fact that $df = (d\sigma + x_3 d\frakn ~|~ \frakn)$,
\[
\begin{split}
\dist(df\circ \P^{-1},\SO(3)) &\le |df\circ \P^{-1} - U|_\euc  \\
&\le |d\sigma\circ \K^{-1} - Q|_\euc  + |x_3| |d\frakn\circ \K^{-1}|_\euc  \\
&=  \dist(d\sigma\circ \K^{-1},\O(2,3)) + x_3\,|d\frakn|_{\g}.
\end{split}
\]
The result now follows by squaring and integrating. 
\end{proof}

The following is an adaptation of \cite[Lemma~5.1]{CM08}, which does not require the surface to be an immersion:

\begin{proposition}[Kirchhoff-Love estimate for $W^{2,2}\cap W^{1,\infty}$-functions]
\label{prop:kirchhoff_love2}
Let $\sigma\in W^{2,2}(\S;\R^3) \cap W^{1,\infty}(\S;\R^3).$
Then, there exists an extension $f\in W^{1,2}(\M^h;\R^3)$ of $\sigma$ such that
\beq
E_{(\M^h,\P)}(f) \le \max\brk{68, 4 \|d\sigma\|_\infty^2} \, \tEP_{(\S,\K)}(\sigma) + Ch^2 \|d\sigma\|_4^4,
\label{eq:KL2}
\eeq
for some $C\ge0$ depending only on $\g$, which can be taken to be zero if $\g$ admits non-trivial parallel vector fields.
Furthermore, the $\|d\sigma\|_\infty$ factor only affects the bending term.
\end{proposition}

\begin{proof}
The proof is similar to that of \propref{prop:kirchhoff_love}, except that $\frakn$ may not even be defined, so we will define an alternative field $\nu\in W^{1,2}(\W;\R^3)\cap L^\infty(\W;\R^3)$ as follows:
Fix a frame $\{v_1,v_2\}$ of $T\S$ which is orthonormal with respect to $\g$, and define
\[
\nu = \Cases{
2\,(d\sigma(v_1) \times d\sigma(v_2))
&|d\sigma(v_1) \times d\sigma(v_2)| < \tfrac12 \\
\displaystyle \frakn
&|d\sigma(v_1) \times d\sigma(v_2)| \ge \tfrac12.
}
\]
By definition, $|\nu| \le 1$. 
Since 
\[
d(d\sigma(v_i)) = (\nabla^{\g} d\sigma)(v_i) + d\sigma(\nabla^{\g}v_i),
\]
and since $|v_i|_{\g} = 1$, 
\[
|d\nu|_{\g} \le 4\,|d\sigma|_{\g}\, \brk{|\nabla^{\g}d\sigma|_{\g} + C|d\sigma|_{\g}},
\]
where $C$ depends only on the norm of $\nabla^{\g}v_i$. 
In particular, if $\g$ admits non-trivial parallel vector field, we could choose $\{v_1,v_2\}$ to be parallel and then $\nabla^{\g}v_i=0$ and the error term vanishes. 

Now let $f(x) = \sigma(x') + x_3 \nu(x')$.
Define $Q\in\O(2,3)$ as in \propref{prop:kirchhoff_love}, i.e., as the projection of $d\sigma\circ \K^{-1}$ onto $\O(2,3)$, and 
\[
U_x = \Big(Q_{x'}~|~Q_{x'}(v_1)\times Q_{x'}(v_2)\Big).
\]
Then, $U\in\SO(3)$, and
\[
\begin{aligned}
\dist(df\circ \P^{-1},\SO(3))
&\le | df\circ \P^{-1} - U| \\
&\le \dist(d\sigma\circ \K^{-1},\O(2,3)) + |x_3|\,|d\nu\circ \K^{-1} |_\euc + |\nu - Q(v_1)\times Q(v_2)|_\euc\\
&= \dist(d\sigma\circ \K^{-1},\O(2,3)) + |x_3|\,|d\nu|_{\g} + |\nu - Q(v_1)\times Q(v_2)|_\euc,
\end{aligned}
\]
where in the passage to the third line we used the fact that $\K:T\S\to\R^2$ is a vector bundle isometry.
Since $\{v_1,v_2\}$ is a $\g$-orthonormal frame, and $\g$ is induced by $\K$,the frame $\{\K v_1,\K v_2\}$ is orthonormal  in $\R^2$.
Thus, by \lemref{lem:cross_half}, if
\[
|d\sigma(v_1)\times d\sigma(v_2)| < \frac12,
\]
then
\[
\dist(d\sigma\circ \K^{-1},\O(2,3)) \ge \sqrt{2}-1.
\]
In this case,
\[
|\nu - Q(v_1)\times Q(v_2)| \le |\nu| + 1 \le \frac{2}{\sqrt{2} -1}\,\dist(d\sigma\circ \K^{-1},\O(2,3)),
\] 
whereas if $|d\sigma(v_1) \times d\sigma(v_2)| \ge \tfrac12$, then
\[
\begin{split}
|\nu - Q(v_1)\times Q(v_2)| &= |\frakn - Q(v_1)\times Q(v_2)| = 0.
\end{split}
\]
We thus have
\[
\dist(df\circ \P^{-1},\SO(3)) \le \frac{\sqrt{2}+1}{\sqrt{2}-1}\dist(d\sigma\circ K^{-1},\O(2,3)) + 4|x_3| \brk{|d\sigma|_{\g}\, |\nabla^{\g}d\sigma|_{\g} + C|d\sigma|_{\g}^2},
\]
Squaring and integrating we obtain the desired result. 
\end{proof}

\subsection{Bounding the 3D model from below}
We now prove the third inequality in \thmref{thm:3D2D}, which immediately follows from the following result:

\begin{theorem}[``Reverse" Kirchhoff-Love estimate]
\label{thm:reverse_KL}
Assume that $\K$ induces a metric with zero Gaussian curvature.
Then, for every $f\in W^{1,2}(\M^h;\R^3)$ there exists a map $\sigma\in W^{2,2}(\S_h;\R^3)$, such that
\[
\tEP_{(\S_{2h},\K)}(\sigma) \le C E_{(\M^h,\calP)}(f)
\]
for some absolute constant $C>0$.
Here, $\S_{2h} = \{x'\in \S ~:~ \dist_{\g}(x', \partial \S)>2h\}$. 
\end{theorem}

This result, as well as the technique, is similar to other results in the literature; see for example \cite[Proof of Theorem 2]{Olb17}.
However we could not pinpoint a specific place in the literature in which it is stated explicitly.
Moreover, the proof given here shows the role of the flatness assumption on the underlying metric in a more vivid way, and so it might be helpful for relaxing this assumption in future work.

We prove \thmref{thm:reverse_KL} in steps, starting with a few lemmas.
In the following, $d(\cdot,\cdot)$ denotes the distance in $\M^h$ with respect to the metric $G$ induced by $\calP$; note that $G$ is also locally-Euclidean.
Furthermore, for two sufficiently close points $x,y\in \M^h$, we define $\Pi_x^y:T_x\M^h\to T_y\M^h$ as the parallel transport of the metric $G$ along the unique shortest geodesic from $x$ to $y$.
Since we assumed $G$ to be smooth, this is well-defined even if $\calP$ is discontinuous. 

\begin{proposition}
\label{prop:reduction}
Let  $f\in W^{1,2}(\M^h;\R^3)$, and let $\eta:\R_+\to[0,1]$ be a smooth function, compactly-supported in $[0,1)$ such that for \emph{some} point $x'_0\in \S_h$,
\[
\frac{1}{h^3} \int_{\M^h} \eta\brk{\frac{d((x'_0,0),y)}{h}}\, \dVol{\calP}(y) = 1.
\]
Define $\sigma:\S_h\to\R^3$ by
\[
\sigma(x') = \frac{1}{h^3} \int_{\M^h} f(y)\, \eta\brk{\frac{d((x',0),y)}{h}}\, \dVol{\calP}(y),
\]
and $Q\in\Gamma(T^*\M^h|_{\S_h}\otimes\R^3)$ by
\[
Q_{x'}(v) = \frac{1}{h^3} \int_{\M^h} df_{y}\circ \Pi_{(x',0)}^y(v) \, \eta\brk{\frac{d((x',0),y)}{h}}\, \dVol{\calP}(y),
\]
for $v\in T^*_{(x',0)}\M^h$.
Then,
\[
d\sigma = Q|_{T\S_h},
\]
and
\[
|\nabla^{\g} d\sigma|_{\g} \le |\nabla^\g Q|_\g.
\]
\end{proposition}

Note that $Q_{x'}$ for $x'\in S_h$ is a linear map between $T_{(x',0)}\M^h\to \R^3$, whereas $d\sigma_{x'}:T_{x'}\S \to \R^3$. 

\begin{proof}
First, we note that 
\[
\frac{1}{h^3} \int_{\M^h} \eta\brk{\frac{d((x',0),y)}{h}}\, \dVol{\calP}(y) = 1
\]
for every $x'\in \S_h$. 
Indeed, since $(\M^h,G)$ is locally-Euclidean, there exists a local coordinate system near $(x',0)$, 
such that $(x',0)$ maps to the origin, and with respect to which the metric is the standard Euclidean metric.
Therefore the $G$-ball of radius $h$ around $(x',0)$ is the Euclidean ball $B_h(0)$ in these coordinates. 
Thus,
\[
\frac{1}{h^3} \int_{\M^h} \eta\brk{\frac{d((x',0),y)}{h}}\, \dVol{\calP}(y) 
= \frac{1}{h^3} \int_{B_h(0)} \eta\brk{\frac{|y|}{h}}\, dy 
\]
is independent of $x'$, and hence equals $1$.

Still in a local Euclidean coordinate system we have
\[
\sigma(x') = \frac{1}{h^3} \int_{\R^3} f(y)\,  \eta\brk{\frac{|(x',0) - y|}{h}}\, dy,
\]
and
\[
Q_i(x') = \frac{1}{h^3} \int_{\R^3} \partial_if(y)\,  \eta\brk{\frac{|(x',0) - y|}{h}}\,dy,
\]
since the local representation of the parallel transport operator is the identity.
We need to show that for $i=1,2$,
\[
\partial_i\sigma(x,y) = Q_i(x,y),
\]
which is an immediate consequence of the commutation of differentiation and mollification in Euclidean space.
The inequality follows from the trivialization of the covariant derivative and the fact that $\nabla^{\g} d\sigma$ only accounts for in-plane variations of the planar components of $Q$.
\end{proof}

The following proposition is an adaptation of  \cite[Lemma~4.1]{LP11}:

\begin{proposition}
\label{prop:df-Q}
Let $\M^h_{2h} = \S_{2h} \times (0,h)$.
The following inequalities hold for some constant $C>0$:
\[
\frac{1}{h} \int_{\M_{2h}^h} |df(x) -  Q_{x'}\circ\Pi_x^{(x',0)}|^2\,\dVol{\calP}(x) \le C E_{(\M^h,\calP)}(f),
\]
\[
\int_{\S_{2h}} |\nabla^\g Q|_\g^2\,\dVol{\K} \le \frac{C}{h^2} E_{(\M^h,\calP)}(f),
\]
and
\[
\int_{\S_{2h}} \dist^2(Q\circ\calP^{-1},\SO(3))\,\dVol{\K} \le C\,  E_{(\M^h,\calP)}(f).
\]
\end{proposition}

\begin{proof}
We start with a local analysis restricted to domains of diameter $O(h)$, hence we may assume once again a Euclidean coordinate system, in which we identify, for $x'\in \S_{2h}$, the $\g$ ball $B_{2h}(x')$ with the Euclidean ball of radius $2h$;
furthermore, we note that the volume form is the Euclidean one under this identification, and that the parallel transport is trivial (the identity matrix).

By the FJM geometric rigidity theorem, there exists a (non-necessarily regular) section $U:\S_{2h}\to \SO(3)$, such that for every $x'\in\S_{2h}$,
\begin{equation}
\label{eq:FJMcylinder}
\int_{B_{2h}(x') \times(0,h)} |df - U_{x'}|^2\,dx \le  
C \int_{B_{2h}(x') \times(0,h)} \dist^2(df ,\SO(3))\,dx,
\end{equation}
where $C>0$ is some absolute constant, and when we identify $B_{2h}(x') \times(0,h)$ with the Euclidean domain via the Euclidean coordinate system.
By the definition of $Q$, we have that
\[
Q_{x'} - U_{x'} = \frac{1}{h^3} \int_{B_h(x') \times(0,h)} (df - U_{x'}) \, \eta\brk{\frac{d((x',0),y)}{h}}\, dy,
\]
hence 
\[
|Q_{x'} - U_{x'}|^2 \le  \frac{C_1}{h^3} \int_{B_h(x') \times(0,h)}  |df - U_{x'}|^2\, dy
\le \frac{C_2}{h^3} \int_{B_{2h}(x') \times(0,h)} \dist^2(df ,\SO(3))\,dy
\]
Let $p',q'\in\S_{2h}$ such that $|p'-q'|\le h$ and let $r' = (p'+q')/2$. Then,
\[
\begin{split}
Q_{p'} - Q_{q'} &= 
\frac{1}{h^3} \int_{\M^h} df \, \brk{\eta\brk{\frac{d((p',0),y)}{h}} - \eta\brk{\frac{d((q',0),y)}{h}}}\,\dVol{\calP}(y)  \\
&= \frac{1}{h^3} \int_{\M^h} (df - U_{r'}) \, \brk{\eta\brk{\frac{d((p',0),y)}{h}} - \eta\brk{\frac{d((q',0),y)}{h}}}\,\dVol{\calP}(y),
\end{split}
\]
from which follows that
\beq
\begin{split}
|Q_{p'} - Q_{q'} |^2 &\le
\frac{C_3}{h^3} \int_{B_{2h}(r')\times(0,h)} |df - U_{r'}|^2\,dy 
\le \frac{C_4}{h^3} \int_{B_{2h}(r')\times(0,h)} \dist^2(df ,\SO(3))\,dy,
\end{split}
\label{eq:Qdiff}
\eeq
Thus,
\[
\begin{split}
\int_{B_h(p')\times(0,h)} |df - Q|^2\,dy &\le
3 \int_{B_h(p')\times(0,h)}\brk{ |df - U_{p'}|^2   + |U_{p'} - Q_{p'}|^2 +  |Q_{p'} - Q|^2}\,dx \\
&\le C_5 \int_{B_{2h}(p') \times(0,h)} \dist^2(df,\SO(3))\,dy.
\end{split}
\]
Similarly, using Cauchy-Schwartz, we have
\[
\begin{split}
|\nabla Q_{q'}|^2 &= \left| \frac{1}{h^4} \int_{B_{h}(q')\times(0,h)} df \, \brk{\nabla \eta\brk{\frac{d((q',0),{y})}{h}}}\, dy \right|^{2}\\
&= \left| \frac{1}{h^4} \int_{B_{h}(ql)\times(0,h)} \left( df - U_{ql} \right) \, \brk{\nabla \eta\brk{\frac{d((q',0),{y})}{h}}}\, dy \right|^{2}\\
&\leq \frac{1}{h^8}\int_{B_{h}(q')\times(0,h)}  \left| df - U_{q'} \right|^{2}  dy \int_{B_{h}(q')\times(0,h)}\left|  \nabla \eta\brk{\frac{d((q',0),{y})}{h}} \right|^{2}\, dy \\
&\le \frac{C_6}{h^5}\int_{B_{h}(q')\times(0,h)}  \left| df - U_{q'} \right|^{2}  dy.
\end{split}
\]
Hence, by equation \eqref{eq:FJMcylinder}, 
\[
\int_{B_h(p')} |\nabla Q|^2\,dx \le \frac{C_7}{h^3} \int_{B_{2h}(p')\times(0,h)} \dist^2(df,\SO(3))\,dx.
\]

Reverting to a coordinate-free notation, we have just shown that
\[
\int_{B_h(x')\times(0,h)} |df(y) - Q_{y'}\circ\Pi_y^{(y',0)}|^2\,\dVol{\calP}(y) \le  C \int_{B_{2h}(x') \times(0,h)} \dist^2(df\circ\calP^{-1},\SO(3))\,\dVol{\calP},
\]
and
\[
\int_{B_h(x')} |\nabla^\g Q|^2\,\dVol{\K} \le \frac{C}{h^3}  \int_{B_{2h}(x') \times(0,h)} \dist^2(df\circ\calP^{-1},\SO(3))\,\dVol{\calP}.
\]
Let $\{x_i\}\subset S_{2h}$ be a maximal $h$-separated set (constructed, for example, by a greedy algorithm).
By covering $\S_{2h}$ by the finitely-many discs $B_h(x_i)$ and summing up these inequalities,  we obtain the first two inequalities, where the constant $C$ is multiplied by at most $25$. 
The third inequality follows from the first, as 
\[
\dist^2(Q_{x'}\circ\calP_{(x',0)}^{-1},\SO(3))  \le 2\, |df(x) - Q_{x'}\circ\Pi_x^{(x',0)}|^2 + 2\, \dist^2(df(x)\circ\calP_{x}^{-1},\SO(3)),
\]
where we used the fact that $\calP$ is $x_3$-independent.
\end{proof}

We now conclude the proof of \thmref{thm:reverse_KL}:
Recall that 
\[
d\sigma = Q|_{T\S_{h}},
\]
and thus, from the structure of $\calP$, we have
\[
d\sigma \circ \K^{-1} = Q\circ \calP^{-1}\circ \iota^\parallel.
\]
where $\iota^\parallel: \R^2\to R^3$ is the standard inclusion map.
Hence,
\[
\dist^2(d\sigma \circ \K^{-1},\O(2,3))  \le \dist^2(Q\circ \calP^{-1},\SO(3)).
\]
Integrating and using \propref{prop:df-Q},
\[
\int_{S_{2h}} \dist^2(d\sigma \circ \K^{-1},\O(2,3)) \,\dVol{\K} \le  C\, E_{(\M,\calP)}(f).
\]
Finally, since, by \propref{prop:reduction}, $|\nabla^{\g}d\sigma| \le |\nabla^\g Q|$, \propref{prop:df-Q} yields the same bound for the bending energy.

\addcontentsline{toc}{section}{References} 
\footnotesize

\begin{thebibliography}{GHKM13}

\bibitem[AKM22]{AKM22}
I.~Alpern, R.~Kupferman, and C.~Maor, \emph{Asymptotic rigidity for shells in
  non-{Euclidean} elasticity}, J. Func. Anal. \textbf{283} (2022), no.~6,
  109575.

\bibitem[CM08]{CM08}
S.~Conti and F.~Maggi, \emph{Confining thin elastic sheets and folding paper},
  Arch. Rat. Mech. Anal. \textbf{187} (2008), 1--48.

\bibitem[Con]{Con23}
S.~Conti, \emph{On the energy of a thin plate with a single dislocation},
  private communication.

\bibitem[COT17]{conti2017symmetry}
S.~Conti, H.~Olbermann, and I.~Tobasco, \emph{Symmetry breaking in indented
  elastic cones}, Math. Models Meth. Appl. Sci. \textbf{27} (2017), no.~02,
  291--321.

\bibitem[ES51]{ES51}
J.~Eshelby and A.N. Stroh, \emph{Dislocations in thin plates}, Philosophical
  Magazine \textbf{42} (1951), no.~335, 1401--1405.

\bibitem[FJM02]{FJM02b}
G.~Friesecke, R.D. James, and S.~M\"uller, \emph{A theorem on geometric
  rigidity and the derivation of nonlinear plate theory from three dimensional
  elasticity}, Comm. Pure Appl. Math. \textbf{55} (2002), 1461--1506.

\bibitem[GHKM13]{GHKM13}
J.~Guven, J.A. Hanna, O.~Kahraman, and M.M. M{\"u}ller, \emph{Dipoles in thin
  sheets}, Europ. Phys. J. E \textbf{36} (2013), 1--10.

\bibitem[KM26]{kupferman2026volterra}
R.~Kupferman and C.~Maor, \emph{From {Volterra} dislocations to strain-gradient
  plasticity}, Calc. Var. PDE \textbf{65} (2026), no.~4, 102.

\bibitem[KMPG25]{KMP25}
R.~Kupferman, C.~Maor, and D.~Padilla-Garza, \emph{The {Willmore} energy and
  curvature concentration}, \url{https://arxiv.org/abs/2511.18982}, 2025.

\bibitem[KS14]{KS14}
R.~Kupferman and J.P. Solomon, \emph{A {Riemannian} approach to reduced plate,
  shell, and rod theories}, J. Func. Anal. \textbf{266} (2014), 2989--3039.

\bibitem[Kup17]{Kup17}
R.~Kupferman, \emph{On the bending energy of buckled edge-dislocations}, Phys.
  Rev. E \textbf{96} (2017), 063002.
  
 \bibitem[Lew23]{Lew23}
M.~Lewicka, \emph{Calculus of Variations on Thin Prestressed Films}, Birkh\"auser Cham, 2023.

\bibitem[Lov27]{Lov27}
A.E.H. Love, \emph{A treatise on the mathematical theory of elasticity}, Fourth
  ed., Cambridge University Press, Cambridge, 1927.

\bibitem[LP11]{LP11}
M.~Lewicka and M.R. Pakzad, \emph{Scaling laws for non-{Euclidean} plates and
  the {$W^{2,2}$} isometric immersions of {Riemannian} metrics}, ESAIM Cont.
  Opt. Calc. Var. \textbf{17} (2011), 1158--1173.

\bibitem[Mao25]{Mao25}
C.~Maor, \emph{On material-uniform elastic bodies with disclinations and their
  homogenization}, Math. Mech. Solids \textbf{30} (2025), no.~9, 2043--2053.
  
\bibitem[MM03]{MM03}
M.G.~Mora and S.~M\"uller, \emph{Derivation of the nonlinear bending-torsion theory for inextensible rods by {$\Gamma$}-convergence},
Calc. Var. PDE \textbf{18} (2003), 287--305.

\bibitem[MO14]{MO14}
S.~M{\"u}ller and H.~Olbermann, \emph{Conical singularities in thin elastic
  sheets}, Calc. Var. PDE \textbf{49} (2014), no.~3-4, 1177--1186.

\bibitem[MS19]{MS19}
C.~Maor and A.~Shachar, \emph{On the role of curvature in the elastic energy of
  non-euclidean thin bodies}, J. Elast. \textbf{134} (2019), no.~2,
  149--173.

\bibitem[Nel02]{Nel02}
D.R. Nelson, \emph{Defects and geometry in condensed matter physics}, Cambridge
  University Press, 2002.

\bibitem[NP87]{nelson1987fluctuations}
D.R. Nelson and L.~Peliti, \emph{Fluctuations in membranes with crystalline and
  hexatic order}, J. Physique \textbf{48} (1987), no.~7, 1085--1092.

\bibitem[Olb16]{olbermann2016energy}
H.~Olbermann, \emph{Energy scaling law for the regular cone}, J. 
  Nonlin. Sci. \textbf{26} (2016), no.~2, 287--314.

\bibitem[Olb17]{Olb17}
\bysame, \emph{Energy scaling law for a single disclination in a thin elastic
  sheet}, Arch. Rat. Mech. Anal. \textbf{224} (2017), no.~3, 985--1019.

\bibitem[Olb18]{olbermann2018shape}
\bysame, \emph{The shape of low energy configurations of a thin elastic sheet
  with a single disclination}, Anal. \& PDE \textbf{11} (2018), no.~5,
  1285--1302.

\bibitem[SN88]{SN88}
H.S. Seung and D.R. Nelson, \emph{Defects in flexible membranes with
  crystalline order}, Phys. Rev. A \textbf{38} (1988), no.~2, 1005.

\bibitem[SZ12]{SZ12}
L.~Scardia and C.I. Zeppieri, \emph{Line-tension model for plasticity as the
  {$\Gamma$}-limit of a nonlinear dislocation energy}, SIAM J.
  Math. Anal. \textbf{44} (2012), no.~4, 2372--2400.

\bibitem[Vol07]{Vol07}
V.~Volterra, \emph{Sur l'{\'e}quilibre des corps {\'e}lastiques multiplement
  connexes}, Ann. Sci.
  Ecole Norm. Sup. Paris \textbf{24} (1907), 401--518.

\bibitem[Wit07]{Wit07}
T.A. Witten, \emph{Stress focusing in elastic sheets}, Rev. Mod.
  Phys. \textbf{79} (2007), no.~2, 643--675.

\end{thebibliography}

\providecommand{\bysame}{\leavevmode\hbox to3em{\hrulefill}\thinspace}
\providecommand{\MR}{\relax\ifhmode\unskip\space\fi MR }
\providecommand{\MRhref}[2]{%
  \href{http://www.ams.org/mathscinet-getitem?mr=#1}{#2}
}
\providecommand{\href}[2]{#2}

\end{document}